\documentclass[10pt,reqno]{amsart}

\usepackage{amsfonts,amssymb,amsmath,mathrsfs}
\usepackage{float}
\usepackage{stmaryrd}
\usepackage{srcltx}
\usepackage{epsfig}
\usepackage{graphicx}
\usepackage{caption}
\usepackage{subcaption}
\usepackage{color}
\usepackage{verbatim}
\usepackage{kotex}

\title[ ]{Cucker-Smale type flocking models on a sphere}

\author[Sun-Ho Choi]{Sun-Ho Choi}
\address[Sun-Ho Choi]{Department of Applied Mathematics and the Institute of Natural Sciences, Kyung Hee University, 1732 Deogyeong-daero, Giheung-gu, Yongin 17104, Republic of Korea}
\email{sunhochoi@khu.ac.kr}

\author[Dohyun Kwon]{Dohyun Kwon}
\address[Dohyun Kwon]{Department of Mathematics, University of Wisconsin-Madison, 480 Lincoln Dr., Madison, WI 53706, USA}
\email{dkwon7@wisc.edu}

\author[Hyowon Seo]{Hyowon Seo}
\address[Hyowon Seo]{Department of Applied Mathematics and the Institute of Natural Sciences, Kyung Hee University, Yongin, 446-701, South Korea}
\email{hyowseo@gmail.com}

\begin{document}

\newtheorem{theorem}{Theorem} [section]
\newtheorem{maintheorem}{Theorem}
\newtheorem{lemma}[theorem]{Lemma}
\newtheorem{proposition}[theorem]{Proposition}
\newtheorem{remark}[theorem]{Remark}
\newtheorem{example}[theorem]{Example}
\newtheorem{exercise}{Exercise}
\newtheorem{definition}{Definition}[section]
\newtheorem{corollary}[theorem]{Corollary}


\newcommand{\noi}{\noindent}
\newcommand{\Z}{\mathbb{Z}}
\newcommand{\R}{\mathbb{R}}
\newcommand{\C}{\mathbb{C}}
\newcommand{\T}{\mathbb{T}}
\newcommand{\bul}{\bullet}
\newcommand{\E}{\mathcal{E}}
\newcommand{\N}{\mathcal{N}}
\newcommand{\RR}{\mathcal{R}}
\newcommand{\D}{\mathcal{D}}
\newcommand{\HH}{\mathcal{H}}

\newcommand{\al}{\alpha}
\newcommand{\dl}{\delta}
\newcommand{\Dl}{\Delta}
\newcommand{\eps}{\varepsilon}
\newcommand{\kk}{\kappa}
\newcommand{\g}{\gamma}
\newcommand{\G}{\Gamma}
\newcommand{\ld}{\lambda}
\newcommand{\lam}{\lambda}
\newcommand{\Ld}{\Lambda}
\newcommand{\s}{\sigma}
\newcommand{\ft}{\widehat}
\newcommand{\wt}{\widetilde}
\newcommand{\cj}{\overline}
\newcommand{\dx}{\partial_x}
\newcommand{\dt}{\partial_t}
\newcommand{\dd}{\partial}
\newcommand{\invft}[1]{\overset{\vee}{#1}}
\newcommand{\lrarrow}{\leftrightarrow}
\newcommand{\embeds}{\hookrightarrow}
\newcommand{\LRA}{\Longrightarrow}
\newcommand{\LLA}{\Longleftarrow}

\newcommand{\wto}{\rightharpoonup}

\newcommand{\jb}[1]
{\langle #1 \rangle}

\definecolor{caribbean green}{rgb}{0.0, 0.8, 0.6}
\newcommand{\dk}[1]{{\color{caribbean green}[#1]}}


\renewcommand{\theequation}{\thesection.\arabic{equation}}
\renewcommand{\thetheorem}{\thesection.\arabic{theorem}}
\renewcommand{\thelemma}{\thesection.\arabic{lemma}}
\newcommand{\bbr}{\mathbb R}
\newcommand{\bbz}{\mathbb Z}
\newcommand{\bbn}{\mathbb N}
\newcommand{\bbs}{\mathbb S}
\newcommand{\bbp}{\mathbb P}
\newcommand{\ddiv}{\textrm{div}}
\newcommand{\bn}{\bf n}
\newcommand{\rr}[1]{\rho_{{#1}}}
\newcommand{\thh}{\theta}
\def\charf {\mbox{{\text 1}\kern-.24em {\text l}}}
\renewcommand{\arraystretch}{1.5}

\thanks{
}

\maketitle

\begin{abstract}
We present a Cucker-Smale (C-S) type flocking model on a sphere. We study velocity alignment on a sphere and prove the emergence of flocking for the proposed model. Our model includes three new terms: a centripetal force, multi-agent interactions on a sphere and inter-particle bonding forces. To compare velocity vectors on different tangent spaces, we introduce a rotation operator in our new interaction term. Due to the geometric restriction, the rotation operator is singular at antipodal points and the relative velocity between two agents located at these points is not well-defined. Based on an energy dissipation property of our model and a variation of Barbalat's lemma, we show the alignment of velocities for an admissible class of communication weight functions. In addition, for sufficiently large bonding forces we prove time-asymptotic flocking which includes the avoidance of antipodal points.

%



\end{abstract}


%
%
\section{Introduction}
\setcounter{equation}{0}

Collective behaviors are ubiquitous phenomena in nature. 
Many organisms employ collective behaviors to survive in nature: examples include the aggregation of bacteria, the flocking of birds and fish schools. Recently, these phenomena have been intensively studied in engineering communities due to their applications. In engineering, a model with a constant speed is preferred for practical reasons and control problems are often addressed. For instance, in \cite{J-K1,J-K2}, Justh and Krishnaprasad considered a unit speed particle model satisfying a Frenet-Serret equation with curvature control. Generalizations of a discrete time Vicsek model with leadership and without leadership were discussed in \cite{J-L-M}.  A planar model that has a gyroscopic steering force with unit speed was studied in \cite{P-L-S}. Their coupling depends on the position and angle of velocity. Leonard et al. designed a particle motion control to collect information in \cite{L-P-L}. They focused on developing and solving an optimal control problem for cost functions.

In the mathematics and physics communities, many researchers  have also studied these phenomena in various perspectives and in various forms.  For examples, Topaz and Bertozzi considered a fluid type model describing social interaction in two spatial dimensions and studied swarming patterns in \cite{T-B}. Their model contains a nonlocal velocity alignment term. In \cite{F-E}, Fetecau and Eftimie presented a discrete velocity model with a nonlocal force and turning rate. They considered the global existence and aggregation phenomena. Especially,  after the mathematical model of Winfree and Kuramoto for collective dynamics \cite{Ku1,Ku2,Wi}, many researchers have proposed  agent-based models to study the emergent behaviors analytically and numerically. Vicsek et al. in \cite{V-C-B-C-S} proposed a second-order flocking model with discrete time scheme, the so-called  Vicsek model. In this model, all agents in the system  have a constant speed. We refer to \cite{D-M1, E-E-M, T-T} for  noise and position dependent force, derivation of a macroscopic model with a diffusion coefficient, and a statistical view point, respectively.

In this paper, we continue to study the flocking model from Cucker and Smale \cite{C-S2}, which is a kind of $N$-body system.  Cucker and Smale \cite{C-S2} introduced a system of ordinary differential equations such that acceleration is  described by  weighted internal relaxation forces:
\begin{align}
\begin{aligned}\label{C-S}
\frac{dx_i}{dt} &= v_i, \\
\frac{dv_i}{dt} &= \sum_{j=1}^{N} \frac{\psi_{ij}}{N}(v_j - v_i),
\end{aligned}
\end{align}
where $x_i$ and $v_i$ are the position and velocity in $\bbr^d$ of the $i$th agent, respectively, and $\psi_{ij}$ is the communication rate between the $i$th and $j$th agents. Cucker and Smale \cite{C-S2} also provided sufficient conditions for initial data leading to flocking configuration. We refer to \cite{A-C-H-L,C-F-R-T,H-L,H-T,M-T,M-T2} for results related to C-S models.

While the original C-S model \eqref{C-S} describes  $N$-body agents in $\bbr^d$, collective behaviors can occur on other manifolds. For example, all living organisms in nature are located in the Earth and the geometry of the Earth may not be negligible in  long-distance travel.
We provide a brief historical remark of flocking dynamics on manifolds.
The Kuramoto model in \cite{Ku1, Ku2} is  a first-order differential equation type model, and the oscillators are arranged in  $\mathbb{S}^1$. Lohe generalized the Kuramoto model to a matrix model in \cite{Lo-1, Lo-2}.  From the matrix model, Lohe derived a dynamical model in $\mathbb{S}^4$ that was inspired by quantum information theory. The ellipsoid model was studied in \cite{zhu}. Vicsek and his collaborators in \cite{V-C-B-C-S} considered the second-order discrete time model with $v\in \mathbb{S}^1$. See \cite{H-J-K} for the continuous time model and \cite{C-H} for general dimension cases.

The main objective of this paper is to derive a C-S type flocking model on a unit sphere. Additionally, we need a new definition of flocking for the model to describe the flocking phenomena on a sphere. The main difficulty comes from analysis of  the velocity difference, a so-called  relative velocity, on a manifold. The concept of a relative velocity on a manifold has been well-developed and widely used in general relativity (See \cite{Car97,McG03,Sch85,Tal08}). The relative velocity can be considered as parallel transport along geodesics (See \cite[Equation (3.109)]{Car97}). On a sphere, a geodesic is a part of a great circle and a parallel transport along a great circle is characterized by a rotation matrix, given in Definition~\ref{mat rot}.
Furthermore, as our domain is not a Euclidean domain but a unit sphere, 
the classical Lyapunov functional approach which provides an exponential convergence rate in the previous articles cannot be applied directly. Instead, we use an integrable system property to obtain the flocking result.

The model proposed in this paper is an agent-based Newton's equation type model as the original C-S model. We assume that each particle has unit mass. Consider an ensemble of agents on a unit sphere. Let $x_i, v_i \in \bbr^3$ be the position and momentum of the $i$th agent respectively. Note that $x_i$ is a unit vector corresponding to the position on the unit sphere.
Because of the geometric restrictions, we need an extra term to conserve the modulus of $x_i$,  called the centripetal force. In the original C-S model, the sum of velocity differences between two particles is considered as acceleration. Since the surface is not flat in our model, we cannot obtain a relative velocity between the other two agents by a simple vector difference. More precisely, if we simply take the difference $v_i-v_j$ between $i$th and $j$th agents as the relative velocity, then 
it may not be contained in the tangent space of the given position vector $x_i$ at time $t>0$. Therefore, we need a new interaction rule between the two velocities $v_i$ and $v_j$. 

The result of this paper are threefold. First, we derive a C-S type system on a unit sphere, based on the original C-S model \eqref{C-S}, by using the centripetal force, a rotation operator, and an inter-particle bonding force. The following system of first-order ordinary differential equations is a counterpart of the C-S model on  the unit sphere: for $x_i, v_i\in \bbr^3$ and $i =1, \ldots, N$,
\begin{align}
\begin{aligned}\label{main}
\dot{x}_i&=v_i,\\
\dot{v}_i&=-\frac{\|v_i\|^2}{\|x_i\|^2}x_i+\sum_{j=1}^N\frac{\psi_{ij}}{N}\big(R_{x_j \shortrightarrow x_i}(v_j)-v_i\big)+\sum_{k=1}^N \frac{\sigma }{N}(\|x_i\|^2x_k - \langle x_i,x_k \rangle  x_i),
\end{aligned}
\end{align}
where $\psi_{ij}=\psi(\|x_i - x_j\|)$ is the communication rate and $\sigma$ is the inter-particle bonding parameter. Here, $R_{x_1\shortrightarrow x_2}(y)$ is the rotation operator, which consists of a matrix $R(x_1, x_2)$ and its matrix product (See Definition \ref{mat rot}). The detailed definition and properties of the operator $R_{x_1\shortrightarrow x_2}(y)$ are  discussed in Section 2.  We refer to \cite{D-S,M-P-T} for C-S models on $\mathbb{T}^d$.


\begin{figure}[h]
  \centering
{\small
\begin{minipage}[t]{5cm}
\centering
     \includegraphics[width=5cm]{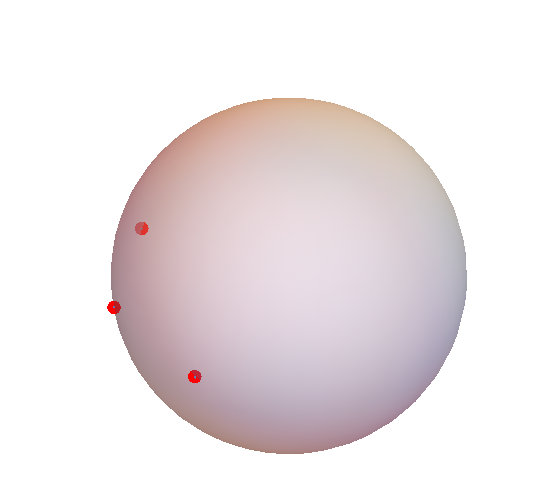}\newline
(a) $t=0$
\end{minipage}
\begin{minipage}[t]{5cm}
\centering
     \includegraphics[width=5cm]{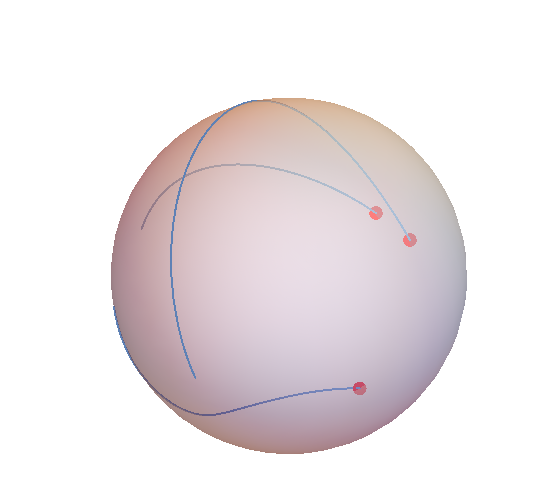}\newline
(b) $t=1$
\end{minipage}
\begin{minipage}[t]{5cm}
\centering
     \includegraphics[width=5cm]{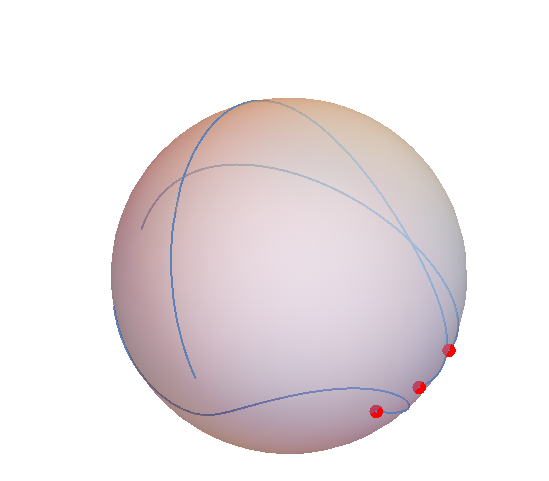}\newline
(c) $t=2$
\end{minipage}


\begin{minipage}[t]{5cm}
\centering
     \includegraphics[width=5cm]{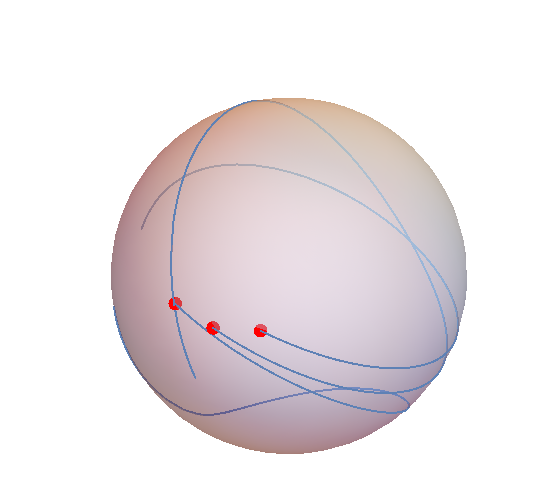}\newline
(d) $t=3$
\end{minipage}
\begin{minipage}[t]{5cm}
\centering
     \includegraphics[width=5cm]{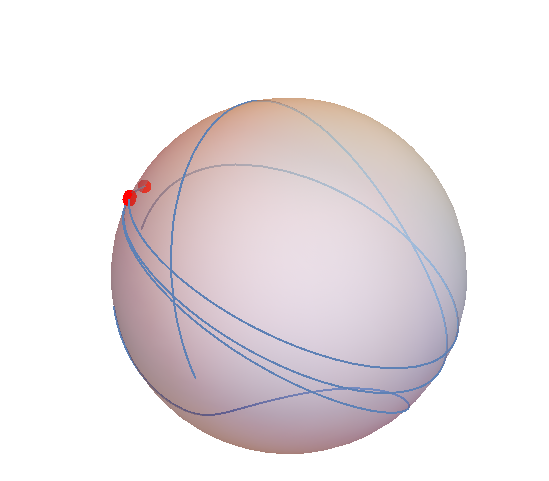}\newline
(e) $t=4$
\end{minipage}
\begin{minipage}[t]{5cm}
\centering
     \includegraphics[width=5cm]{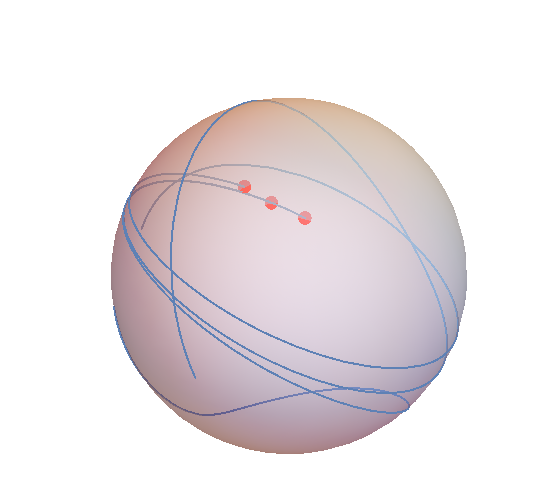}\newline
(f) $t=5$
\end{minipage}

\begin{minipage}[t]{5cm}
\centering
     \includegraphics[width=5cm]{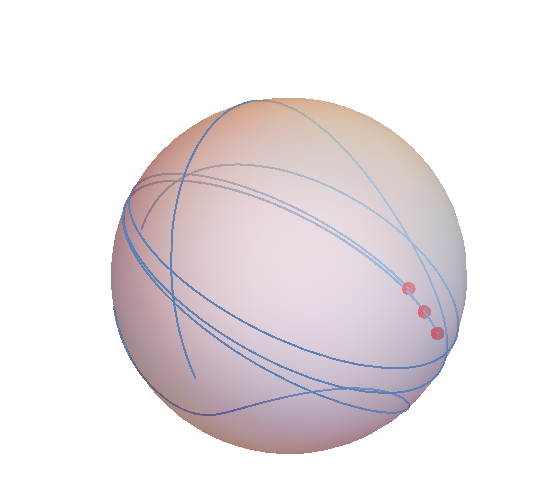}\newline
(g) $t=6$
\end{minipage}
\begin{minipage}[t]{5cm}
\centering
     \includegraphics[width=5cm]{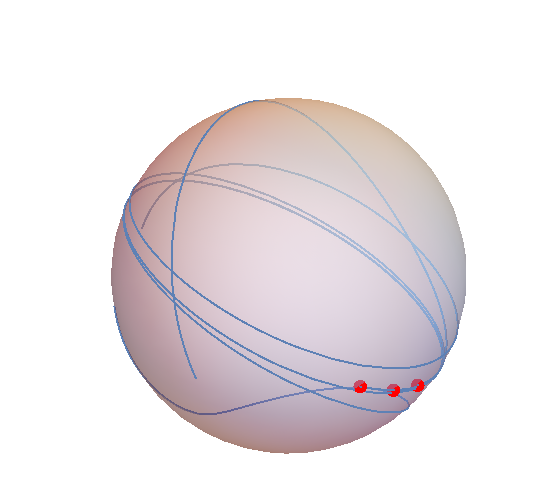}\newline
(h) $t=7$
\end{minipage}
\begin{minipage}[t]{5cm}
\centering
     \includegraphics[width=5cm]{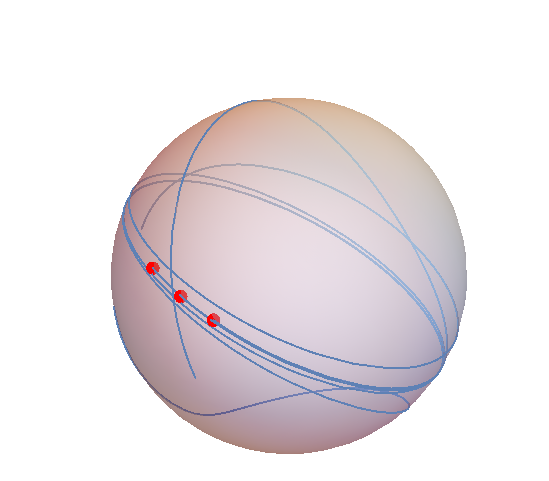}\newline
(i) $t=8$
\end{minipage}

%
%
\caption{Numerical experiment of flocking on a unit sphere for three particles.}
\label{fig1}
}
\end{figure}



In the original C-S model \cite{C-S2}, the communication rate $\psi_{ij}=\psi(\|x_i-x_j\|)$ quantifies how the agents affect each other and $\psi$ is a decreasing function of distance $\|x_i-x_j\|$ between two agents $x_i$ and $x_j$. The main concern in system \eqref{main} is to determine $\psi_{ij}$ when two points $x_i$ and $x_j$ in $\mathbb{S}^2$ are antipodal. As there are infinitely many geodesics connecting two antipodal points in $\mathbb{S}^2$, it is unclear what effect the corresponding antipodal point has. Rather, it is natural to assume that the influence of one agent on another agent is negligible if their positions are antipodal. To illustrate this, we assume that $\psi_{ij}=\psi(\|x_j - x_i\|)$ and
\begin{align}
\label{eqn:psi}
\hbox{ a decreasing $C^1$ function } \psi : [0,2] \to [0, +\infty) \hbox{ satisfies } \psi(2) = 0 \hbox { and } \psi'(2)<0.
\end{align}

Second, we develop flocking on a unit sphere as follows.
\begin{definition}
\label{def:flo}
A system has time-asymptotic flocking on a sphere in $\R^3$ if and only if a solution $(x_i,v_i)_{i=1}^N $ of the system satisfies the following condition:
\begin{itemize}
\item (velocity alignment) the relative velocity on the unit sphere goes to zero time-asymptotically:
\begin{align}\nonumber
\lim_{t \rightarrow \infty} \max_{1\leq i,j\leq N} \|x_i(t) + x_j(t)\| \|  R _{x_j(t)\shortrightarrow x_i(t)}(v_j(t)) - v_i(t) \| = 0.
\end{align}

\item (antipodal points avoidance) any two agent are not located at the antipodal points:
\[\liminf_{t\geq 0}\min_{1\leq i,j\leq N}\|x_i(t)+x_j(t)\|>0.\]
\end{itemize}
\end{definition}

It is worth noting that Definition~\ref{def:flo} has the avoidance of antipodal points while the boundedness of position fluctuations is included in the original definition of flocking in $\bbr^3$. As the domain is compact in our case, the boundedness condition for the position difference $x_i-x_j$ does not guarantee the formation of a group. Instead, the avoidance of antipodal points is required because antipodal points are as far away from each other as possible.

The rotation operator from one point to its antipodal point cannot be defined (See Remark~\ref{rem:22}(2)). For this reason, the relative velocity will be only considered if  points are not antipodal. In fact, the conditions below
\begin{align}\nonumber
\lim_{t \rightarrow +\infty} \|x_i(t) + x_j(t) \| \|  R _{x_j(t)\shortrightarrow x_i(t)}(v_j(t)) - v_i(t) \|^k = 0
\end{align}
are equivalent for any $k >0$ as long as $v_i$ and $v_j$ are uniformly bounded in time (See Lemma~\ref{lem:equ}). The simplest example of flocking on a sphere is a set of unit speed circular motions.
\begin{example}
\label{ex:1}
For $t \geq 0$ and $1 \leq i \leq N$, consider
\begin{align}
\label{eqn:ex1}
x_i(t) := (\cos(t+\alpha_i), \sin(t+\alpha_i), 0), \quad v_i(t) := (-\sin(t+\alpha_i), \cos(t+\alpha_i), 0) \hbox{ and } 0 \leq \alpha_i < \pi.
\end{align}
By direct computation, it holds that for all $t \geq 0$ and $1 \leq i \leq N$,
\begin{align}
\nonumber 
R _{x_j(t)\shortrightarrow x_i(t)}(y) =
\left(\begin{matrix}
\cos(\alpha_i - \alpha_j) & -\sin(\alpha_i - \alpha_j) &0  \\
\sin(\alpha_i - \alpha_j) & \cos(\alpha_i - \alpha_j)  &0 \\
0 & 0 & 1
\end{matrix}\right)
y.
\end{align}
Thus, we have $R _{x_j(t)\shortrightarrow x_i(t)} (v_j(t)) = v_i(t)$ for all $t \geq 0$, and $(x_i,v_i)_{i=1}^N$ given in \eqref{eqn:ex1} satisfies the conditions in Definition~\ref{def:flo}.
\end{example}

Third, we provide the global-in-time existence and  flocking result for an admissible class of communication weight functions. To obtain the flocking estimate for the solution to \eqref{main}, we consider the following total energy functional $\E$ given by the sum of the kinetic energy $\E_K$ and configuration energy $\E_C$ motivated by \cite{PKH10}: for a given ensemble $(x_i,v_i)_{i=1}^N$,  we define energy functional $\E(x(t),v(t))$ such as
\begin{align}
\label{eqn:e}
\E:= \E_K + \E_C, \quad \E_K(t):= \frac{1}{N}\sum_{k=1}^N\|v_k(t)\|^2, \quad  \E_C(t) := \frac{\sigma}{2N^2 } \sum_{k,l=1}^N \| x_k(t) - x_l(t) \|^2.
\end{align}
If the bonding force rate $\sigma$ is large enough comparing the differences of agents' velocities and positions, then the following flocking result holds. What follows is a summary of our results from Theorem~\ref{thm:wel}, Theorem~\ref{thm:43} and Theorem ~\ref{thm 4.8}.
\begin{maintheorem}
\label{thm:0}
For $\psi$ satisfying \eqref{eqn:psi} and $\sigma \geq 0$, there exists a unique global solution to \eqref{main} and the system in \eqref{main} has the velocity alignment on a sphere.
Moreover, for $\sigma > N^2\E(0)/2$, the solution to \eqref{main} has time-asymptotic flocking on a unit sphere. Here, $\E(0)$ is the initial energy of the system given in \eqref{eqn:e}.
\end{maintheorem}

We note that it remains open to show the emergence of flocking for $\sigma = 0$ and the complete position flocking for $\sigma>0$. For the original C-S model defined in a flat space, the a priori assumption that the spatial diameter of agents is uniformly bounded yields an exponential decay for the maximum velocity differences between agents. Conversely, the exponential decay of the maximum velocity differences also leads to the uniform boundedness of the position difference. From this a priori estimate argument, the emergence of the flocking for the original C-S model is attained. However, this standard methodology is not applicable  to our model on the sphere.

The main difficulty comes from estimating the position difference between two agents. We expect that the exponential decay of the relative velocity $\|R_{x_j \shortrightarrow x_i}(v_j)-v_i\|$ yields the boundedness of the position difference. However, in our case since the sphere is a curved space, $R_{x_j \shortrightarrow x_i}$ depends on agents' location. Thus, unlike the original C-S model, it is not clear whether there is a Gronwall-type dissipative differential inequality for the relative velocity, even assuming a priori that the position difference at $t=0$ is small.

The rest of this paper is organized as follows. In Section 2, we present a derivation of the C-S type model \eqref{main} on the unit sphere from the original C-S model. In Section 3, we provide the global well-posedness of the solution to the derived model \eqref{main}. In Section 4, we prove the asymptotic flocking theorem for the system.  Finally, Section 5 is devoted to the summary of our main results.  \newline

{\bf Notation:} For given $z\in \bbr^3$, we use the symbols  $\|z\|$ and  $\|z\|_{\infty}$ to denote the $\ell_2$-norm and $\ell_{\infty}$-norm, respectively. For three-dimensional vectors $y$ and $z$, we denote the standard inner product between $y$ and $z$ as $\langle y,z \rangle$.\newline

\section{Motivation and derivation of the C-S type model on a sphere}
\setcounter{equation}{0}
\subsection{Derivation}
In this section, we present a derivation of the C-S type flocking model on a sphere. After normalization, we can assume that the domain is a unit sphere:
\begin{align*}
\mathbb{S}^2=\{(a_1,a_2,a_3)\in \bbr^3: a_1^2+a_2^2+a_3^2=1 \}.
\end{align*}
We first consider the classical form of the C-S model in $\bbr^3$.
\begin{align*}
\dot{x}_i&=v_i,\\
\dot{v}_i&=\sum_{j=1}^N \frac{\psi_{ij}}{N}(v_j-v_i).
\end{align*}
Here, $x_i$ and $v_i$ represent the position and velocity of the $i$th agent, respectively, and $\psi_{ij}=\psi(\|x_i-x_j\|)$ is the communication rate or weight function for the interaction between $i$th and $j$th agents. In the original C-S model, the communication rate $\psi$ has a key role in the emergence of flocking as a control parameter. Note that  to obtain global flocking result in the original C-S model, the power of denominator is important, i.e., if $\psi(x)=1/(1+|x|^2)^{\beta/2}$, the following statement holds.

\begin{enumerate}
\item $\beta\in [0,1]$, unconditional flocking occurs,
\item $\beta\in (1,\infty)$, conditional flocking occurs.
\end{enumerate}
\begin{remark}
Unlike in the case of $\bbr^3$, the flocking model on the unit sphere does not require the long-range communications for unconditional flocking. On the other hand, since it has singularity in the antipodal positions, a vanishing condition in the antipodal positions such as \eqref{eqn:psi} is required.
\end{remark}

In the following, we construct a C-S type flocking model on a unit sphere that  shares the same structure as the original C-S model, including the communication rate and velocity difference. Essentially, the C-S flocking model consists of  three components: (1) the classic relation between the position $x_i$ and velocity $v_i$ in the first equation, $\dot{x}_i = v_i$, (2) the velocity difference $v_j-v_i$ in the second equation, and (3) the communication rate $\psi_{ij}$.

For the first trial, we fix the first two 
components and we try to find $\psi_{ij}$ to conserve the modulus of $x_i$. To make dynamics on the sphere, we need to obtain
\begin{align}\label{eq 1.0}
\|x_i\|\equiv 1,\quad \mbox{for all time}~t>0.
\end{align}
If we have compatible initial data $x_i(0)$ and $v_i(0)$, an equivalent relation to the above is
\begin{align}\label{eq 1.05}
\langle x_i,v_i\rangle \equiv 0,\quad \mbox{for all time}~t>0.
\end{align}
Under this compatible initial condition, we can also obtain another equivalent relation as follows:\phantom{\eqref{eq 1.05}}
\begin{align}
\label{equiv rel}
\langle x_i(t),\dot{v}_i(t)\rangle +\|v_i(t)\|^2\equiv 0,\quad \mbox{for all time}~t>0,
\end{align}
and
\begin{align}
\label{equiv rel1}
 \langle x_i(0),v_i(0)\rangle=0.
\end{align}

For the details, see Proposition \ref{prop 3.1}.

\begin{proposition}Let $\psi_{ij}$ be  scalar functions depending on $\{x_1, \cdots, x_N\}$ for all $i,j\in \{1,\ldots,N\}$. Then a solution to the C-S flocking model \eqref{C-S} with communication rate $\psi_{ij}$ lying on the unit sphere does not exist.
\end{proposition}
\begin{proof}
We assume that each agent lies on a unit sphere, i.e., for all $1\leq i\leq N$,
\begin{align*}x_i\in\mathbb{S}^2. \end{align*}
Assume that the solution $(x_i,v_i)_{i=1}^N$ to the following C-S type equation is as follows:
\begin{align*}
\dot{x}_i&=v_i,\\
\dot{v}_i&=\sum_{j=1}^N\frac{\psi_{ij}}{N}(v_j-v_i),
\end{align*}
and take the inner product between $\dot{v}_i$ and $x_i$ to obtain
\begin{align*}
\langle \dot{v}_i, x_i\rangle&=\sum_{j=1}^N\frac{\psi_{ij}}{N}\langle v_j-v_i~,x_i\rangle.
\end{align*}
Note that from the argument in \eqref{eq 1.0}-\eqref{equiv rel1}, we have
\begin{align*}
\langle x_i,v_i\rangle \equiv 0
\end{align*}
and
\begin{align*}
\langle x_i,\dot{v}_i\rangle +\|v_i\|^2\equiv 0.
\end{align*}
It follows that
\begin{align*}
-\|v_i\|^2=\langle \dot{v}_i, x_i\rangle=\sum_{j=1}^N\frac{\psi_{ij}}{N}\langle v_j-v_i~,x_i\rangle=\sum_{j=1}^N\frac{\psi_{ij}}{N}\langle v_j~,x_i\rangle.
\end{align*}
As $\psi_{ij}$ does not depend on $\{v_1, \cdots, v_N\}$, the above equation holds for any $v_i$ such that $\langle v_i , x_i \rangle = 0$. Therefore, there is no such  solution $\psi_{ij}$ to the equation.
\end{proof}

If we do not change the first equation, the only possibility to construct a unit sphere model is modification of the $v_j-v_i$ terms for $ i,j \in \{1,\ldots,N\}$. As mentioned in Section 1, from geometrical consideration, it is natural to consider a relative velocity to keep all agents' positions within the given manifold. However, without interactions between each agent, the positions of each agent are not maintained  in the manifold. Therefore, we need an additional external force term. In the sense of \eqref{eq 1.0} - \eqref{equiv rel1}, the acceleration $\dot{v}_i$ in our model has to satisfy
\begin{align*}
\langle x_i,\dot{v}_i\rangle = -\|v_i\|^2.\end{align*}
We propose one possible model as follows. We add a self-consistency term, the so-called centripetal force, as $\displaystyle -\frac{\|v_i\|^2}{\|x_i\|^2}x_i$.
Consider the following centripetal equation:
\begin{align*}\dot{v}_i = -\frac{\|v_i\|^2}{\|x_i\|^2}x_i.\end{align*}

It is well known that the solution to the centripetal equation gives  uniform circular motion on a sphere. Then, the remaining interaction term for the $i$th agent between agents must be orthogonal to $x_i$. Thus, we need an operator map from the tangent space of $x_j$ to the tangent space of  $x_i$.

\subsection{Model}

Our proposed model is as follows: for $x_i, v_i\in \bbr^3$ and $i =1, \ldots, N$,
\begin{align}
\begin{aligned}\nonumber
\dot{x}_i&=v_i,\\
\dot{v}_i&=-\frac{\|v_i\|^2}{\|x_i\|^2}x_i+\sum_{j=1}^N\frac{\psi_{ij}}{N}\big(R_{x_j\shortrightarrow x_i}(v_j)-v_i\big)+\sum_{k=1}^N \frac{\sigma }{N}(\|x_i\|^2x_k - \langle x_i,x_k \rangle  x_i).
\end{aligned}
\end{align}
Here, $R_{x_1\shortrightarrow x_2}(y)=R(x_1,x_2)\cdot y$ is the rotation operator defined below and $\sigma > 0$ is the rate of the inter-particle bonding force.

\begin{definition}\label{mat rot}Let  $x_1,x_2\in \mathbb{S}^2$ be column vectors with $x_1\ne -x_2$.  We define a $3\times 3$ matrix
  \begin{align}\nonumber R(x_1,x_2) :=\left\{ \begin{aligned}
  &\langle x_1 , x_2\rangle I - x_1 x_2^T + x_2 x_1^T + (1-  \langle x_1, x_2\rangle) \Big( \frac{x_1 \times x_2}{|x_1 \times x_2|} \Big) \Big( \frac{x_1 \times x_2}{|x_1 \times x_2|} \Big)^T,\quad&\mbox{if}\quad x_1\ne x_2,\\
  &I,\quad&\mbox{if}\quad x_1=x_2,
  \end{aligned}
  \right.\end{align}
where $I$ is the identity matrix in $\mathbb{R}^3$ and $M^T$ is the transpose of a matrix $M$.
The operation $R_{x_1\shortrightarrow x_2}(y)$ is defined by the linear transform as
\begin{align*}
R_{x_1\shortrightarrow x_2}(y)=R(x_1, x_2)\cdot y.
\end{align*}
Here, $y$ is a column vector and $\cdot$ is the matrix product. Furthermore, we set
\begin{align}\nonumber
\|x_1 + x_2\| R(x_1, x_2) := 0 \hbox{ at } x_1 = -x_2.
\end{align}

\end{definition}

\begin{remark}
\label{rem:22}
\begin{enumerate}
\item By direct calculation, we can  rewrite the rotation matrix as
\begin{align}\label{rotating}
R(x_1, x_2)=\cos \theta ~I+\sin\theta ~[{\bf u}]+(1-\cos \theta)~ {\bf u}\otimes {\bf u},
\end{align}
  where $\displaystyle {\bf u}=\frac{x_1\times x_2}{\|x_1\times x_2\|}$ is the normalized cross product between the vectors $x_1$ and $x_2$, $\theta$ is the angle between $x_1$ and $x_2$, and $[{\bf u}]$, ${\bf u}\otimes {\bf u}$ are given by
\begin{align*}
[{\bf u}]=\left(\begin{matrix}
0&-u_{e_3}&u_{e_2}\\u_{e_3}&0&-u_{e_1}\\-u_{e_2}&u_{e_1}&0
\end{matrix}\right) ,\quad {\bf u}\otimes {\bf u}=\left(\begin{matrix}
u_{e_1}^2&u_{e_1} u_{e_2}&u_{e_1} u_{e_3}
\\
u_{e_1} u_{e_2}&u_{e_2}^2&u_{e_2} u_{e_3}
\\
u_{e_1}u_{e_3}&u_{e_2} u_{e_3}&u_{e_3}^2
\end{matrix}\right).
\end{align*}
Here, $u_{e_1}$, $u_{e_2}$ and $u_{e_3}$ are the first, second and third components of the vector $u$, respectively.

\item For $x_1=-x_2$, the third term in the right-hand side of \eqref{rotating} is not well-defined. Geometrically, if two points are located  at opposite poles (for example, the north pole and south pole), then there are infinitely many geodesics connecting the two points. This means that we cannot determine a unique parallel transport.  From this property, we need to take $\psi_{ij}=\psi(\|x_i-x_j\|)$ with some decay assumption at $x_1=-x_2$.  Thus, the model \eqref{main} is  valid although $R(x_1,x_2)$ is not defined at $x_1=-x_2$. Our assumptions on $\psi_{ij}$ and the well-posedness will be shown later.

\item  For any $x_1 \in \mathbb{S}^2$, it holds that
\begin{align*}\lim_{x_2\to x_1} R(x_1,x_2)=I.\end{align*}

\item Note that $R(x_1,x_2)$ is bounded for any $x_1,x_2\in \mathbb{S}^2$ (See Lemma~\ref{lemma 2.4}). Thus, although $R(x_1,x_2)$ is not defined at $x_1= - x_2$, it is natural to define $\|x_1+x_2\|R(x_1,x_2)=0$ at $x_1 = - x_2$.
\end{enumerate}\end{remark}

Next, we provide elementary properties of the rotation matrix.

\begin{lemma}\label{lemma 2.3}
For $x_1, x_2 \in \mathbb{S}^2$ with $x_1 \ne -x_2$, the following holds.
\begin{align}
\nonumber
R_{x_1\shortrightarrow x_2}(x_1) = x_2, \quad R_{x_1\shortrightarrow x_2}(x_2) = 2\langle x_1, x_2\rangle x_2 - x_1~  \hbox{ and }  ~R_{x_1\shortrightarrow x_2} (x_1 \times x_2) = x_1 \times x_2.
\end{align}
Furthermore, we have
\begin{align}
\nonumber
R_{x_1\shortrightarrow x_2}^{T} = R_{x_2\shortrightarrow x_1},\quad  
R_{x_1\shortrightarrow x_2}^T \circ R_{x_1\shortrightarrow x_2} = I_{\mathbb{S}^2}.
\end{align}

\end{lemma}

\begin{proof}
By definition of the rotation map $R_{x_1\shortrightarrow x_2} : \mathbb{S}^2 \shortrightarrow \mathbb{S}^2$, it is enough to consider the equivalent properties for a matrix $R(x_1,x_2)$ defined in Definition \ref{mat rot}. For the case of $x_1=x_2$, we can easily check that $R(x_1,x_2)$ satisfies Lemma~\ref{lemma 2.3}, since $R(x_1,x_2)=I$.

Next, we consider the case, $x_1\ne x_2$. As the two vectors $x_1$ and $x_2$ are perpendicular to $x_1\times x_2$, direct computation shows that
\begin{align}
\label{eqn:1rot}
R(x_1,x_2)\cdot x_1 = x_2, \quad R(x_1,x_2)\cdot x_2 = 2\langle x_1, x_2\rangle x_2 - x_1  \hbox{ and }  R(x_1,x_2) \cdot (x_1 \times x_2) = x_1 \times x_2.
\end{align}
Furthermore, since we have
\begin{align}\begin{aligned}\nonumber
R(x_1,x_2)^T &= \langle x_1 , x_2\rangle I - x_2 x_1^T + x_1 x_2^T + (1-  \langle x_1 , x_2\rangle) \left( \frac{x_1 \times x_2}{|x_1 \times x_2|} \right) \left( \frac{x_1 \times x_2}{|x_1 \times x_2|} \right)^T,
\end{aligned}
\end{align}
we conclude that
\begin{align}
\label{eqn:2rot}
R(x_1,x_2)^{T} = R(x_2,x_1). 
\end{align}

We show that $R(x_1,x_2)$ is an orthogonal matrix, that is
\begin{align}
\label{eqn:rot11}
R(x_1,x_2)^T R(x_1,x_2) = I.
\end{align}
From \eqref{eqn:1rot} and \eqref{eqn:2rot}, it holds that
\begin{equation*}
R(x_1,x_2)^T R(x_1,x_2)\cdot x_1 = R(x_2,x_1)\cdot x_2 = x_1
\end{equation*}
and
\begin{equation*}
R(x_1,x_2)^T R(x_1,x_2) \cdot x_2 =R(x_2,x_1)\cdot(2\langle x_1, x_2\rangle x_2 - x_1) = 2\langle x_1, x_2\rangle x_1 - (2\langle x_1, x_2\rangle x_1 - x_2) = x_2.
\end{equation*}
Furthermore, the  last equality in \eqref{eqn:1rot} implies that
\begin{equation*}
R(x_1,x_2)^TR(x_1,x_2)\cdot (x_1\times x_2) = x_1\times x_2.
\end{equation*}
As $x_1, x_2$ and $x_1\times x_2$ are linearly independent, we conclude that for any $y \in \R^3$
\begin{align*}
\big(R(x_1,x_2)^TR(x_1,x_2) - I\big)\cdot y = 0,
\end{align*}
and thus we conclude \eqref{eqn:rot11}.
\end{proof}

As a consequence of Lemma \ref{lemma 2.3}, we have the following property.

\begin{lemma}\label{lemma 2.4}
Let  $x_1, x_2 \in \mathbb{S}^2$ with $x_1 \ne -x_2$. Then, we have
\begin{align}
\nonumber
\|v\|=\|R_{x_1\shortrightarrow x_2}(v)\| \hbox{ for any } v \in \mathbb{R}^n,
\end{align}
where $R_{x_1\shortrightarrow x_2}(y)$ is the rotation operator from Definition \ref{mat rot}.
\end{lemma}

\begin{proof}
From Lemma \ref{lemma 2.3}, it holds that
\begin{align*}
\|v\|^2=\langle v,v\rangle&=\langle v,R(x_1,x_2)^T R(x_1,x_2)\cdot v\rangle\\
&=\langle R(x_1,x_2)\cdot v, R(x_1,x_2)\cdot v\rangle
=\langle R_{x_1\shortrightarrow x_2}(v),R_{x_1\shortrightarrow x_2}(v)\rangle=
\|R_{x_1\shortrightarrow x_2}(v)\|^2.
\end{align*}
\end{proof}

\begin{remark}
\begin{enumerate}
\item For $x_1, x_2 \in \mathbb{S}^2$ with $x_1\ne -x_2$, the rotation map $R_{x_1\shortrightarrow x_2} : \mathbb{S}^2 \shortrightarrow \mathbb{S}^2$ is a  linear bijection and isometry between two spheres. Furthermore, the differential of $R_{x_1\shortrightarrow x_2}$ at $x_1$ gives a linear map between the tangent spaces at $x_1$ and $x_2$, which contain the velocity vectors. Note that the differential of the map $R_{x_1\shortrightarrow x_2}$ at $x_1$ is the same as the matrix $R(x_1,x_2)$ since the rotation map is linear.

\item The rotation operator is a rather admissible choice. If we assume that $\|x_i\|\equiv 1$, then we have \begin{align*}\langle v_i,x_i\rangle=0.\end{align*}
    When the force equation has  ~$\displaystyle -\frac{\|v_i\|^2}{\|x_i\|^2}x_i$ term, then we have to replace the $v_j$ term by some of the tangential vectors of the sphere at $x=x_i$. In this point of view, the rotation vector $R_{x_j\shortrightarrow x_i}(v_j)$ is the most natural choice for replacement.
\end{enumerate}

\end{remark}

Lastly, our model includes the inter-particle bonding force. For the original C-S model in $\bbr^3$, the definition of the flocking contains the uniform boundedness of position differences. This property has been proven for the unconditional and conditional cases using the corresponding initial data and  means that the velocity alignment is faster than the dissipation of the ensemble. This is achieved by obtaining the exponential decay rate of velocity difference and from the exponential decay, the uniform boundedness of position differences was obtained.  However, for $\bbs^2$ case in this paper, we cannot control the diameter of agent's positions, even with an exponential decay rate of velocity difference, due to the geometric property of $\bbs^2$ and the rotation operator. Thus, the argument that used in the original C-S model cannot be applied  to our model. To guarantee the antipodal points avoidance, we included the inter-particle bonding force similar to one for the augmented C-S model in \cite{PKH10},
\[\displaystyle \frac{\sigma}{N} \sum_{k=1}^N (x_k - x_i).\]
We note that tighter spatial configurations can be achieved by adding this term to the original C-S model \cite{PKH10}. As the agent needs to be located on the sphere for all time in our case, so we need some modification in the above terms and we develop the inter-particle bonding force on a sphere based on Lohe operator in \cite{Lo-1,Lo-2}:
\begin{align}
\sum_{k=1}^N \frac{\sigma }{N}(\|x_i\|^2x_k - \langle x_i,x_k \rangle  x_i).
\end{align}
From this modification, we can prove that the ensemble $(x_i,v_i)_{i=1}^N$  satisfying \eqref{main} is located on the sphere and obtain an energy dissipation property that plays a crucial role in the proof of the flocking theorem. For the detailed, see Proposition  \ref{prop 3.1} and \ref{prop 2.6}.
\section{The global well-posedness of a unit sphere model}
\setcounter{equation}{0}
In this section, we prove the global existence and uniqueness of the solution to \eqref{main}.

\begin{theorem}
\label{thm:wel} If $\psi$ satisfies \eqref{eqn:psi}, then there exists a unique solution $(x_i,v_i)_{i=1}^N$ to the system \eqref{main} for all time. In particular, $(x_i)_{i=1}^N$ are located in a unit sphere for all time $t>0$.
\end{theorem}

Note that 
in our model, the position of each agent is located on a unit sphere. The following natural conditions are required for this property to appear: We say that the initial data are admissible if it holds that
\begin{align}
\label{eqn:adm}
\langle v_i(0),x_i(0)\rangle=0 \quad \mbox{and}\quad  \|x_i(0)\|= 1  \quad \mbox{for all} \quad i\in \{ 1,\ldots,N \}.
\end{align}
The following proposition shows that the modulus of $x_i$ is conserved and $v_i$ is in the tangent space of a unit sphere at $x_i$.

\begin{proposition}\label{prop 3.1}
Let $(x_i(t),v_i(t))_{i=1}^N $ be a solution to \eqref{main} and assume that the initial data are admissible and $\psi_{ij}$ are nonnegative bounded functions for all $i,j\in \{1,\ldots,N\}$. Then  for all $i\in \{1,\ldots,N\}$ and  $t>0$,
\begin{align*}\langle v_i(t),x_i(t)\rangle=0 \hbox{ and } \quad \|x_i(t)\|= 1.\end{align*}
\end{proposition}
\begin{proof}

We take the inner product between $\dot{x}_i$ and $x_i$. From the first equation of \eqref{main}, it follows that
\begin{align*}
\frac{d}{dt}\| x_i\|^2=2\langle \dot{x}_i,x_i\rangle =2\langle v_i,x_i\rangle.
\end{align*}
Thus, the necessary and sufficient condition for conservation of the modulus $\|x_i(t)\|\equiv 1$ is
\begin{align*}\langle v_i(t),x_i(t)\rangle\equiv 0,\end{align*} since  initial conditions satisfy $\|x_i(0)\|=1$ and $\langle v_i(0),x_i(0)\rangle=0$, $i\in \{1,\ldots,N\}$.

Note that
\begin{align*}0=\frac{d}{dt}\langle v_i,x_i\rangle=\langle \dot{v}_i,x_i\rangle+\langle v_i,\dot{x}_i\rangle=\langle \dot{v}_i,x_i\rangle+\langle v_i,v_i\rangle.\end{align*}
Thus,  an equivalent relation to $\langle v_i,x_i\rangle\equiv 0$ is
 \begin{center}
 $\langle x_i,\dot{v}_i\rangle +\|v_i\|^2\equiv 0.$\footnote{Therefore, we have \eqref{equiv rel}.}
 \end{center}
From the above argument, it suffices to prove that $\langle x_i,\dot{v}_i\rangle +\|v_i\|^2\equiv 0$.
Taking the inner product between  the second equation on the system and $x_i$ leads that
\begin{align*}
\langle \dot{v}_i, x_i\rangle &=-\|v_i\|^2+\sum_{j=1}^N\frac{\psi_{ij}}{N}(\langle R_{x_j\shortrightarrow x_i}(v_j), x_i\rangle-\langle v_i, x_i\rangle)
+\sum_{k=1}^N \frac{\sigma }{N}(\|x_i\|^2\langle x_k, x_i\rangle  - \langle x_i,x_k \rangle \langle x_i,x_i\rangle)
\\&=-\|v_i\|^2+\sum_{j=1}^N\frac{\psi_{ij}}{N}(\langle R_{x_j\shortrightarrow x_i}(v_j), x_i\rangle-\langle v_i, x_i\rangle)
.
\end{align*}
By Lemma \ref{lemma 2.3}, the operator $R_{x_j\shortrightarrow x_i}$ satisfies
\begin{align*}R_{x_j\shortrightarrow x_i}^T=R_{x_j\shortrightarrow x_i}^{-1}=R_{x_i\shortrightarrow x_j},\end{align*}
and $R_{x_i\shortrightarrow x_j}(x_i)=x_j$.

The above equalities yield
\begin{align*}
\langle \dot{v}_i, x_i\rangle &=-\|v_i\|^2+\sum_{j=1}^N\frac{\psi_{ij}}{N}(\langle R_{x_j\shortrightarrow x_i}(v_j), x_i\rangle-\langle v_i, x_i\rangle)\\
&=-\|v_i\|^2+\sum_{j=1}^N\frac{\psi_{ij}}{N}( \langle v_j,R_{x_i\shortrightarrow x_j}(x_i)\rangle-\langle v_i, x_i\rangle)\\
&=-\|v_i\|^2+\sum_{j=1}^N\frac{\psi_{ij}}{N}(\langle v_j, x_j\rangle-\langle v_i, x_i\rangle).
\end{align*}
Thus, we have
\begin{align}\label{eq 2.7}
\langle \dot{v}_i, x_i\rangle +\|v_i\|^2=\sum_{j=1}^N\frac{\psi_{ij}}{N}(\langle v_j, x_j\rangle-\langle v_i, x_i\rangle).
\end{align}
We sum up \eqref{eq 2.7} with respect to index $i$ to obtain
\begin{align*}
\frac{d}{dt}\sum_{i=1}^N|\langle v_i,x_i\rangle|^2&=2\sum_{i=1}^N ( \langle \dot{v}_i,x_i\rangle+\langle v_i,\dot{x}_i\rangle)\langle v_i,x_i\rangle  \\
&=  2\sum_{i=1}^N(\langle \dot{v}_i, x_i\rangle +\|v_i\|^2)~\langle v_i,x_i\rangle
\\
&= 2\sum_{i=1}^N \sum_{j=1}^N\frac{\psi_{ij}}{N}(\langle v_j, x_j\rangle-\langle v_i, x_i\rangle)~\langle v_i,x_i\rangle
\\
&\leq  2\sum_{i=1}^N \sum_{j=1}^N\frac{\psi_{ij}}{N}\langle v_j, x_j\rangle~\langle v_i,x_i\rangle.
\end{align*}
Then we have
\begin{align*}
\frac{d}{dt}\sum_{i=1}^N |\langle v_i(t),x_i(t)\rangle|^2&\leq2N \max_{1\leq j,k\leq N}\frac{\psi_{jk}}{N} \sum_{i=1}^N \big|\langle v_i(t), x_i(t)\rangle \big|^2.
\end{align*}
Since we assume that initial data satisfy $\sum_{i=1}^N |\langle v_i(0),x_i(0)\rangle|^2=0$, the Gronwall inequality implies that
\begin{align*}\sum_{i=1}^N |\langle v_i(t),x_i(t)\rangle|\equiv0,~\mbox{for}~ t>0.\end{align*}
As a consequence, the above argument shows that
\begin{align*}\|x_i(t)\|\equiv 1, \quad \mbox{for}~t>0,~ i\in\{1,\ldots N\} . \end{align*}
\end{proof}

 As we mentioned before, we cannot use the Lyapunov functional approach to obtain the flocking estimate due to the geometric modification. 
Instead, the following integrability in Proposition \ref{prop 2.6} plays an important role in the proof of the flocking theorem.

\begin{proposition}\label{prop 2.6}
Let $(x_i(t),v_i(t))_{i=1}^N $ be a solution to \eqref{main}. Assume that  $(x_i(0),v_i(0))_{i=1}^N $ satisfies the admissible initial data condition in \eqref{eqn:adm} and  $\psi_{ij}$ are nonnegative bounded  functions and $\psi_{ij}=\psi_{ji}$ for all $i,j\in \{1,\ldots,N\}$. Then we have
\begin{align}
\label{eqn:prop26}
\frac{d\E}{dt}=-\sum_{i,j=1}^N\frac{\psi_{ij}}{N^2}\| R_{x_j\shortrightarrow x_i}(v_j)-v_i\|^2.
\end{align}
Moreover,  the following estimate holds:
\begin{align}
\label{eqn:prop261}
\mathcal{V}(t)\leq \sqrt{N \E(0)},
\end{align}
where $\mathcal{V} : [0,+\infty) \rightarrow [0,+\infty)$ is the maximal speed given by
\begin{align}
\label{eqn:prop262}
\mathcal{V}(t):=\max_{1\leq i\leq N}\|v_i(t)\|.
\end{align}

\end{proposition}
\begin{proof}
We take the time-derivative of $\E$ and use \eqref{main} to obtain
\begin{align*}\begin{aligned}
\frac{d\E}{dt} &= \frac{2}{N}\sum_{i=1}^N \langle \dot v_i , v_i \rangle + \frac{\sigma}{N^2} \sum_{i,j=1}^N  \langle x_i - x_j , v_i - v_j \rangle\\
&= -\frac{2}{N}\sum_{i=1}^N\frac{\|v_i\|^2}{\|x_i\|^2}\langle x_i, v_i \rangle+ \frac{2}{N^2}\sum_{i,j=1}^N \psi_{ij} \langle R_{x_j\shortrightarrow x_i}(v_j) - v_i , v_i \rangle  \\&\qquad+  \frac{2\sigma}{N^2}\sum_{i,j=1}^N  \langle x_j - x_i ,  v_i \rangle +  \frac{\sigma}{N^2}\sum_{i,j=1}^N  \langle x_i - x_j ,  v_i - v_j \rangle.
\end{aligned}
\end{align*}

By interchanging indices $i$ and $j$, we have
\begin{align*}
2 \sum_{i,j=1}^N  \langle x_j - x_i , v_i \rangle = \sum_{i,j=1}^N  \langle x_j - x_i, v_i \rangle + \langle x_i - x_j , v_j \rangle = - \sum_{i,j=1}^N \langle x_i - x_j, v_i - v_j \rangle.
\end{align*}

From Proposition \ref{prop 3.1}, it holds that $\langle x_i(t), v_i(t) \rangle=0$ and thus
\begin{align}\label{eq 2.6}
\frac{d\E}{dt} &= \frac{2}{N^2}\sum_{i,j=1}^N \psi_{ij} \langle R_{x_j\shortrightarrow x_i}(v_j) - v_i , v_i \rangle  .
\end{align}

From the symmetric assumption  $\psi_{ij}=\psi_{ji}$, it follows that
\begin{align*}
\frac{d\E}{dt}=\sum_{i,j=1}^N\frac{\psi_{ij}}{N^2}\Big(\langle R_{x_i\shortrightarrow x_j}(v_i)-v_j,v_j\rangle+\langle R_{x_j\shortrightarrow x_i}(v_j)-v_i,v_i\rangle\Big).
\end{align*}
Note that from the properties of operator $R_{x_i\shortrightarrow x_j}$ in Lemmas \ref{lemma 2.3} and \ref{lemma 2.4},
\begin{align*}\langle R_{x_i\shortrightarrow x_j}(v_i),v_j\rangle=\langle v_i,R_{x_j\shortrightarrow x_i}(v_j)\rangle ,\quad \langle v_j,v_j\rangle=\langle R_{x_j\shortrightarrow x_i}(v_j),R_{x_j\shortrightarrow x_i}(v_j)\rangle.
\end{align*}
Therefore, we conclude \eqref{eqn:prop26}.

Next, we consider \eqref{eqn:prop261}. By the definition of the energy functional $\E$, for each time $t>0$, we have
\begin{align*}
\mathcal{V}^2(t)\leq N\E_K(t)\leq N\E(t),
\end{align*}
for $\mathcal{V}$ given in \eqref{eqn:prop262}.
From \eqref{eqn:prop26}, we obtain
\begin{align*}
\mathcal{V}^2(t)\leq N\E(0),
\end{align*}
we conclude \eqref{eqn:prop261}.
\end{proof}

\begin{remark}
The energy dissipation in \eqref{eqn:prop26} holds if  $R_{x_i\shortrightarrow x_j}$ satisfies
\begin{equation}\label{essential}
R_{x_i\shortrightarrow x_j}^{-1} = R_{x_i\shortrightarrow x_j}^{T} = R_{x_j\shortrightarrow x_i}.
\end{equation}
Therefore, even if we choose another sphere transformation $T$, if $T$ satisfies \eqref{essential}, then the energy dissipation in \eqref{eqn:prop26} also holds.\end{remark}

\begin{lemma}
\label{lem:con}
For $Q:= (\R^3 \setminus \{0\}) \times (\R^3 \setminus \{0\}) \times \R^3$, a function $T : Q \rightarrow \R^3$ defined by
\begin{align}
\label{eqn:1con}
T(x_1,x_2,v) =\left \{
\begin{aligned}
&\psi \Big( \Big\| \frac{x_1}{\|x_1\|} - \frac{x_2}{\|x_2\|}\Big\| \Big) R_{\small \frac{x_2}{\|x_2\|}\shortrightarrow \frac{x_1}{\|x_1\|}}(v),\quad  &\hbox{ if }  \frac{x_1}{\|x_1\|} + \frac{x_2}{\|x_2\|} \neq 0,\\
&0,\quad  &\hbox{ if } \frac{x_1}{\|x_1\|} + \frac{x_2}{\|x_2\|} = 0
\end{aligned}\right.
\end{align}
is locally Lipschitz continuous in $Q$. Here, $R$ is the rotation operator given in Definition \ref{mat rot} and $\psi$ satisfies assumptions in Theorem~\ref{thm:0}.
\end{lemma}

\begin{proof}
Let
\begin{align}\nonumber
A:= \left\{ (x_1, x_2, v) \in Q : \frac{x_1}{\|x_1\|} = \frac{x_2}{\|x_2\|} \right\} \hbox{ and } B:= \left\{ (x_1, x_2, v) \in Q : \frac{x_1}{\|x_1\|} = -\frac{x_2}{\|x_2\|} \right\} .
\end{align}
We claim that $T$ is Lipschitz in small neighborhood of $(x_1^*, x_2^*, v^*) \in Q$.

First, we consider the case that $(x_1^*, x_2^*, v^*) \in Q \setminus (A \cup B)$. Note that the first three terms of $\psi R$ for the rotation operator $R$ are Lipschitz continuous in $Q$. It remains to show that the last term is locally Lipschitz continuous in $Q$.  We denote the last term of $\psi R$ by $VV^T v$, where $V: Q \to \R^n$ is given by
\begin{align}
\nonumber
V(x_1, x_2,v) =\left \{
\begin{aligned}
&\psi\left(\left\|\frac{x_1}{\|x_1\|}-\frac{x_2}{\|x_2\|}\right\|\right)^{\frac{1}{2}} \Big(1 - \Big\langle \frac{x_1}{\|x_1\|} , \frac{x_2}{\|x_2\|} \Big\rangle \Big)^{\frac{1}{2}} \frac{x_1 \times x_2}{\| x_1 \times x_2 \|},\quad  &\hbox{ in } Q\setminus(A\cup B),\\
&0,\quad  &\hbox{ in } A \cup B.
\end{aligned}\right.
\end{align}
In a small neighborhood of $~(x_1^*, x_2^*, v^*) \in Q \setminus (A \cup B)$, $~\| x_1 \times x_2\|$ is nonzero and  $V(x_1,x_2)V(x_1,x_2)^T v$ is Lipschitz continuous with respect to $x_1$, $x_2$ and $v$.

Next, we consider the second case, $(x_1^*, x_2^*, v^*) \in A$. As the previous step, we focus on the last term of $\psi R$ by $VV^T v$. By the definition of $V$,
\begin{align}
\label{eqn:con12}
\| V(x_1, x_2,v) \| \leq C_1 \left(1 - \left\langle \frac{x_1}{\|x_1\|} , \frac{x_2}{\|x_2\|} \right\rangle \right)^{\frac{1}{2}} = \frac{C_1}{\sqrt{2}} \left\| \frac{x_1}{\|x_1\|} - \frac{x_2}{\|x_2\|} \right\|,
\end{align}
for some constant $C_1>0$. Furthermore, \eqref{eqn:con12} and $  V(x_1^*, x_2^*,v^*) = 0$ at $x_1^*/\|x_1^*\| = x_2^*/\|x_2^*\|$ imply that
\begin{align}
\begin{aligned}
\label{eqn:con13}
\| V(x_1, x_2,v) - V(x_1^*, x_2^*,v^*) \| &\leq \frac{C_1}{\sqrt{2}}  \left\| \frac{x_1}{\|x_1\|} - \frac{x_1^*}{\|x_1^*\|} - \frac{x_2}{\|x_2\|} + \frac{x_2^*}{\|x_2^*\|} \right\|\\&  \leq \frac{C_2 \| (x_1, x_2) - (x_1^*, x_2^*) \|}{\min\{ \|x_1^*\|, \|x_2^*\|\}}
\end{aligned}
\end{align}
for some constant $C_2>0$. Here, we used the following simple inequality:
for any $a,b \in \R^n \setminus \{0\}$,
\begin{align}\nonumber
\left\| \frac{a}{\|a\|} - \frac{b}{\|b\|} \right\| \leq \left\| \frac{a-b}{\|a\|} \right\| + \left\| \frac{b}{\|a\|} - \frac{b}{\|b\|} \right\| \leq  \frac{2\|a-b\|}{\|a\|}.
\end{align}

We choose a small ball of center $(x_1^*, x_2^*,v^*)$ in $Q \setminus B$. For any two points $(x_1, x_2, v)$ and $(y_1, y_2, w)$ in the ball, if either $(x_1, x_2, v)$ or $(y_1, y_2, w)$ in $A$, we apply \eqref{eqn:con13} to obtain the Lipschitz constant
\begin{align}\nonumber
\frac{C_2}{\min\{ \|x_1\|, \|x_2\|\}} \hbox{ or } \frac{C_2}{\min\{ \|y_1\|, \|y_2\|\}}.
\end{align}
If both $(x_1, x_2, v)$ or $(y_1, y_2, w)$ are in $Q \setminus (A \cup B)$, we consider the line segment $\alpha (x_1, x_2, v) + (1- \alpha) (y_1, y_2, w)$ for all $\alpha \in [0,1]$. If the line segment intersects with $A$, then we apply \eqref{eqn:con13} with the triangle inequality to obtain the Lipschitz constant. Thus, it is enough to consider
\begin{align}\nonumber
\{  \alpha (x_1, x_2, v) + (1- \alpha) (y_1, y_2, w) : 0 \leq \alpha \leq 1 \} \subset Q \setminus (A \cup B).
\end{align}
We point out that $V$ is differentiable at $(x_1, x_2, v)$ in $Q \setminus (A \cup B)$ and
\begin{align}\nonumber
\left\| \frac{\partial V}{\partial x_1} (x_1, x_2, v) \right\|, \quad \left\| \frac{\partial V}{\partial x_2} (x_1, x_2, v) \right\| \leq \frac{C_0}{\min\{\|x_1\|, \|x_2\|\}}
\end{align}
for some constant $C_0>0$. By the fundamental theorem of calculus, 
we have
\begin{align}
\label{eqn:con10}
\| V(x_1, x_2, v) - V(y_1, y_2, w) \| \leq \frac{2C_0\| (x_1, x_2, v) - (y_1, y_2, w)\|}{\displaystyle\min_{i =1,2}\min_{0 \leq \alpha \leq 1} \| \alpha x_i + (1- \alpha) y_i\| } ,
\end{align}
and we conclude that $V$ is Lipschitz in the small ball of center $(x_1^*, x_2^*,v^*)$. 


\medskip

Lastly, we consider the case that $(x_1^*, x_2^*, v^*) \in B$.
We denote $(x_1^*, x_2^*) = (k_1a,-k_2a)$ for some $a \in \mathbb{S}^2$ and $k_1, k_2 >0$. 
Since $\displaystyle \frac{1}{2}\Big\|\frac{x_1}{\|x_1\|} + \frac{x_2}{\|x_2\|} \Big\|^2 = 1+\Big\langle \frac{x_1}{\|x_1\|}, \frac{x_2}{\|x_2\|} \Big\rangle$ and  $\| x_1 \times x_2\|^2 = \|x_1\|^2 \|x_2\|^2 - \langle x_1, x_2 \rangle^2$, it holds that
\begin{align}\begin{aligned}
\label{eqn:con23}
\bigg\| \Big(1 - \Big\langle \frac{x_1}{\|x_1\|} , \frac{x_2}{\|x_2\|} \Big\rangle \Big)^{\frac{1}{2}} \frac{x_1 \times x_2}{\| x_1 \times x_2\|}\bigg\| & =
\bigg\|\frac{\big(\|x_1\|\|x_2\| - \langle x_1 ,x_2 \rangle \big)^{\frac{1}{2}} }{\sqrt{\|x_1\|\|x_2\|}} \frac{x_1 \times x_2}{\| x_1 \times x_2\|}\bigg\| \\
& =
\bigg\|\frac{1}{\sqrt{\|x_1\|\|x_2\|}}  \frac{x_1 \times x_2}{
\big(\|x_1\|\|x_2\| + \langle x_1 ,x_2 \rangle \big)^{\frac{1}{2}}
}\bigg\| \\
&=  \sqrt{2} ~\bigg\|\frac{ \big(\frac{x_1}{\|x_1\|} + \frac{x_2}{\|x_2\|}\big) \times \frac{x_2}{\|x_2\|}}{\| {\frac{x_1}{\|x_1\|} + \frac{x_2}{\|x_2\|}}\|} \bigg\|\\& \leq \sqrt{2}.
\end{aligned}
\end{align}
From \eqref{eqn:con23} and Lemma~\ref{lem:psi} below, we have
\begin{align}
\label{eqn:con24}
\left\| V(x_1, x_2, v) \right\| \leq C_3 \left\| \frac{x_1}{\|x_1\|} + \frac{x_2}{\|x_2\|} \right\|
\end{align}
for some constant $C_3>0$. Thus, $V(k_1a, -k_2 a) = 0$ and \eqref{eqn:con24} imply
\begin{align}
\begin{aligned}
\label{eqn:con31}
\| V(k_1a, -k_2a,v^*) - V(x_1, x_2,v)\| &\leq C_3\bigg\| \frac{x_1}{\|x_1\|}+ \frac{x_2}{\|x_2\|} \bigg\|\\
&=C_3\left\| \frac{x_1}{\|x_1\|} - \frac{k_1 a}{\|k_1 a\|} + \frac{x_2}{\|x_2\|} - \frac{- k_2 a}{\|-k_2 a\|} \right\|\\
&\leq \frac{C_4 \| (k_1a, -k_2a) - (x_i, x_j) \|}{\min\{k_1,k_2\}}
\end{aligned}
\end{align}
for some constant $C_4>0$, which implies that \begin{align*}\| V(k_1a, -k_2a, v^*) - V(x_1, x_2,v)\|  \rightarrow 0 \quad \mbox{as} \quad  (x_1, x_2,v) \rightarrow (k_1a, -k_2a,v^*).\end{align*}
Thus $V$ is continuous at $(k_1 a,-k_2a, v^*)$. As discussed in the second case, we consider a small ball of center $(x_1^*, x_2^*, v^*)$ in $Q\setminus A$ and apply \eqref{eqn:con10} and \eqref{eqn:con31}. This shows that $\psi VV^T$ and $T$ are Lipschitz continuous in a small ball of center $(x_1^*, x_2^*, v^*)$. 
\end{proof}

\begin{lemma}
\label{lem:psi}
For $\psi$ satisfying \eqref{eqn:psi}, there exists $C>0$ such that for all $x, y \in \mathbb{S}^2$
\begin{align}
 \frac{\| x + y \|^2}{C}\leq \psi(\|x - y\|) \leq C \| x + y \|^2.
\end{align}
\end{lemma}
\begin{proof}
From the Lipschitzness of $\psi$ and $\psi(2) = 0$, we have $\psi(a) \leq C_1(2-a)$ for $a \in [0,2]$. On the other hand, as $x, y \in \mathbb{S}^2$,
\[2 - \|x - y\| = \frac{\|x + y\|^2}{2+\|x - y\|} \leq \frac{1}{2}\|x + y\|^2.\]
As a consequence, we conclude that
\begin{align*}
\psi(\|x - y\|) \leq C_1(2-\|x - y\|) \leq \frac{C_1}{2} \|x + y\|^2.
\end{align*}
Similarly, from \eqref{eqn:psi}, there exists $C_2>0$ such that $\psi(a) \geq C_2(2-a)$ for all $a \in [0,2]$. Note that
 \[2 - \|x - y\| = \frac{\|x + y\|^2}{2+\|x - y\|} \geq \frac{1}{4}\|x + y\|^2.\] Thus, we conclude that
\begin{align*}
\psi(\|x - y\|) \geq C_2(2-\|x - y\|) \geq \frac{C_2}{4} \|x + y\|^2.
\end{align*}

\end{proof}

We are ready to prove the existence and uniqueness theorem.

\begin{proof}[Proof of Theorem~\ref{thm:wel}]
We consider the following system of ordinary differential equations:
\begin{align}
\begin{aligned}\label{eqn:wel11}
\dot{x}_i&=v_i,\\
\dot{v}_i&=-\frac{\|v_i\|^2}{\|x_i\|^2}x_i+\sum_{j=1}^N\frac{\psi_{ij}}{N}(T(x_i, x_j,v_j)-v_i)+\sum_{k=1}^N \frac{\sigma }{N}(\|x_i\|^2x_k - \langle x_i,x_k \rangle  x_i),
\end{aligned}
\end{align}
where $T(\cdot, \cdot,\cdot)$ is given in \eqref{eqn:1con}.
From the parallel argument of Proposition~\ref{prop 3.1}, if $(x_i,v_i)_{i=1}^N$ is a solution of \eqref{eqn:wel11}, then $\{x_i\}_{i=1}^N \subset \mathbb{S}^2$. Therefore, $(x_i,v_i)_{i=1}^N$ is also a solution of \eqref{main}. It remains to show that the solution to \eqref{eqn:wel11} uniquely exists for all time $t>0$.

For the admissible initial data, Lemma~\ref{lem:con} implies that the right-hand side of \eqref{eqn:wel11} is Lipschitz continuous with respect to $(x_i,v_i)_{i=1}^N$ in a small neighborhood of $(x_i(0),v_i(0))_{i=1}^N$ in $\R^{6N}$. From the Picard-Lindel\"{o}f Theorem (See \cite[Theorem 2.2]{Tes12}), a solution of \eqref{eqn:wel11} exists in an interval $[0,\epsilon]$ for some $\epsilon>0$.

We now prove that the solution of \eqref{eqn:wel11} uniquely exists for all $t>0$. Suppose that $I_M=[0,t_{M})$ is the maximal interval of existence of a solution starting at $t=0$. Proposition~\ref{prop 3.1} implies that $(x_i)_{i=1}^N \subset \mathbb{S}^2$. Furthermore, from the energy inequality in Proposition~\ref{prop 2.6}, we conclude that $(v_i)_{i=1}^N$ are uniformly bounded for all time $t>0$. As $(x_i,v_i)_{i=1}^N$ are uniformly bounded, we can apply the extensibility of solutions in \cite[Corollary 2.2]{Tes12} and conclude that $t_{M} =  \infty$.
\end{proof}
~
\newline
\section{Flocking theorem}
\setcounter{equation}{0}
In this section, we prove the flocking theorem for the unit sphere model.  Because of the curved geometry, we adapt a vector difference $v_j-v_i$ of the original C-S model to $R_{x_j\shortrightarrow x_i}(v_j)-v_i$. Mainly, this new term causes difficulty in analyzing the asymptotic behavior of the solution to the C-S model on a unit sphere. As we mentioned before, we cannot use the Lyapunov functional approach that is  used in the previous articles to obtain the flocking theorem for $\bbr^d$.   Our plan for the proof of the main theorem contains the following three steps. First, we prove the modified version of  Barbalat's lemma in which an integrable function with bounded derivative converges to zero. Then, Proposition \ref{prop 2.6} implies that
$\psi_{ij}\| R_{x_j\shortrightarrow x_i}(v_j)-v_i\|^2$
is integrable with respect to time. Finally, from Lemma~\ref{lemma 3.4}, we verify that the above quantity has a bounded time derivative, and this fact combined with the Barbalat type lemma yields the velocity alignment result for $\sigma\geq 0$ and the flocking estimate for $\sigma>0$ under a sufficient condition for the initial data. We notice that  we need to control the diameter of position difference to obtain flocking behavior. To control the diameter, we consider the additional bonding force under the initial data condition: $2\sigma>N^2\E(0)$.

The following lemma provides our main framework for the velocity alignment.

\begin{lemma}\label{lemma 3.3}
Suppose that a continuous nonnegative function $f : [0,\infty) \to \mathbb{R}$ satisfies
\begin{align}
\label{eqn:1lem33}
\lim_{b\to \infty}\int_a^b f(\tau)d\tau<\infty
\end{align}
for some constant $a\in \bbr$. Assume that on the support $\{t\in [a,\infty):f(t)>0\}$, $f$ has uniformly bounded time-derivative: that is, there is a constant $C>0$ such that
\begin{align}
\label{eqn:11lem33}
\Big|\frac{df}{dt}\Big|<C \quad \mbox{on}~ \quad \{t\in [a,\infty):f(t)>0\}.
\end{align}
Then 
\begin{align}
\label{eqn:2lem33}
\lim_{t\to \infty} f(t)=0.
\end{align}

\end{lemma}

We  postpone the proof until Appendix~\ref{ap:0}. Next, we estimate the time-derivative of $R$.
\begin{lemma}\label{lemma 4.3}
For any $(x_i,v_i)_{i=1}^N$ satisfying $\dot{x}_i=v_i$ for all $i \in \{1,2, \cdots, N\}$, we assume that
\begin{align*}\|x_i\|=1,\quad \langle v_i,x_i\rangle=0,\quad  \|v_i\|\leq \mathcal{V}_{max} \hbox{ and } ~x_i  \ne - x_j \hbox{ for any } i,j \in \{1,2, \cdots, N\}.\end{align*}
Then, the following holds for some constant $C>0$.
\begin{align*}
\Big\|\frac{dR(x_j,x_i)}{dt}\Big\|\leq C \mathcal{V}_{max}+\frac{C}{\| x_i+x_j\|} \mathcal{V}_{max}.
\end{align*}
\end{lemma}

We also postpone the proof until Appendix~\ref{ap:1}, as it is technical.

\begin{lemma}\label{lemma 3.4}Let $(x_i,v_i)_{i=1}^N $ be the solution to system \eqref{main} subject to the admissible initial data $(x_i(0),v_i(0))_{i=1}^N $. We assume that   
$\psi_{ij}=\psi_{ji}$ and
\begin{align*}x_i \neq - x_j,~ \hbox{ for all } i,j \in \{1,2, \cdots, N\}.\end{align*}
Then, it holds that
\begin{align*}
\Big|\frac{d}{dt}\| R_{x_j\shortrightarrow x_i}(v_j)-v_i\|^2
\Big|&\leq
C\Big((N\E(0))^{3/2} +\frac{(N\E(0))^{3/2}}{\| x_i+x_j\|}
+\max_{1\leq i,k\leq N} \psi_{ik}N\E(0)
\Big).
\end{align*}
Here, $C>0$ is a generic constant and $\E$ is given in \eqref{eqn:e}.
\end{lemma}
\begin{proof}
Note that
\begin{align*}
\frac{1}{2}\frac{d}{dt}\| R_{x_j\shortrightarrow x_i}(v_j)-v_i\|^2
&=\Big\langle  \frac{d(R(x_j,x_i)\cdot v_j)}{dt} -\frac{dv_i}{dt},  R(x_j,x_i)\cdot v_j-v_i\Big\rangle
\\
&=
\Big\langle \frac{dR(x_j,x_i)}{dt}\cdot v_j+ R(x_j,x_i)\cdot \frac{dv_j}{dt}-\frac{dv_i}{dt},  R(x_j,x_i)\cdot v_j-v_i\Big\rangle
\\
&=
\Big\langle \frac{dR(x_j,x_i)}{dt}\cdot v_j, R(x_j,x_i)\cdot v_j-v_i\Big\rangle\\
&
\quad+\Big\langle  R(x_j,x_i)\cdot \frac{dv_j}{dt}, R(x_j,x_i)\cdot v_j-v_i\Big\rangle
\\&\quad-\Big\langle  \frac{dv_i}{dt}, R(x_j,x_i)\cdot v_j-v_i\Big\rangle =:I^{ij}_1+I^{ij}_2+I^{ij}_3.
\end{align*}

In the sequel, we will obtain estimates for  $I^{ij}_1$, $I^{ij}_2$ and $I^{ij}_3$, separately. We first consider the $I^{ij}_3$ case. From the second equation of \eqref{main}, we can rewrite $I^{ij}_3$ as follows:
\begin{align}
\begin{aligned}\label{eqn 3.3}
I^{ij}_3&=-\Big\langle  \frac{dv_i}{dt},  R_{x_j\shortrightarrow x_i}(v_j)-v_i\Big\rangle\\
&=-\Big\langle  -\frac{\|v_i\|^2}{\|x_i\|^2}x_i+\sum_{k=1}^N\frac{\psi_{ik}}{N}(R_{x_k\shortrightarrow x_i}(v_k)-v_i)
, ~ R_{x_j\shortrightarrow x_i}(v_j)-v_i\Big\rangle.
\end{aligned}
\end{align}
Note that by Proposition \ref{prop 3.1},
\begin{align}\label{eqn 3.1}
x_i\perp v_i.
\end{align}
From the properties of $R_{x_j\shortrightarrow x_i}$ in Lemma \ref{lemma 2.3}, it follows that
\begin{align}\label{eqn 3.2}
 \langle x_i,  R_{x_j\shortrightarrow x_i}(v_j)\rangle= \langle  R_{x_i\shortrightarrow x_j}(x_i), v_j\rangle= \langle  x_j, v_j\rangle=0.
\end{align}
By \eqref{eqn 3.1} and \eqref{eqn 3.2}, we simplify \eqref{eqn 3.3} as follows:
\begin{align*}
I^{ij}_3&=-\Big\langle  \sum_{k=1}^N\frac{\psi_{ik}}{N}(R_{x_k\shortrightarrow x_i}(v_k)-v_i)
,  R_{x_j\shortrightarrow x_i}(v_j)-v_i\Big\rangle.
\end{align*}
We take the absolute value of the above and use the Cauchy and triangle inequalities   to obtain
\begin{align*}
|I^{ij}_3|&=\bigg|\Big\langle  \sum_{k=1}^N\frac{\psi_{ik}}{N}(R_{x_k\shortrightarrow x_i}(v_k)-v_i)
,  R_{x_j\shortrightarrow x_i}(v_j)-v_i\Big\rangle\bigg|\\
&\leq  \max_{1\leq l,m \leq N } \psi_{lm} \max_{1\leq k \leq N}|\langle  R_{x_k\shortrightarrow x_i}(v_k)-v_i
,  R_{x_j\shortrightarrow x_i}(v_j)-v_i\rangle|
\\&\leq   \max_{1\leq l,m \leq N } \psi_{lm} \max_{1\leq k \leq N}\| R_{x_k\shortrightarrow x_i}(v_k)-v_i\|~\|
 R_{x_j\shortrightarrow x_i}(v_j)-v_i\|
 \\&\leq   \max_{1\leq l,m \leq N } \psi_{lm} \max_{1\leq k \leq N} (\| R_{x_k\shortrightarrow x_i}(v_k)\|+\|v_i\|)~(\|
 R_{x_j\shortrightarrow x_i}(v_j)\|+\|v_i\|).
 \end{align*}

We note that by Lemma \ref{lemma 2.4}, the rotation operator $R_{x_j\shortrightarrow x_i}$ conserves the modulus of velocities $v_j$.
\begin{align*}\| R_{x_k\shortrightarrow x_i}(v_j)\|=\|v_j\|,~ \mbox{for any}~ i,j,k \in \{1,\ldots,N\}.\end{align*}
Thus, we have
\begin{align*}
|I^{ij}_3|&\leq   \max_{1\leq l,m \leq N } \psi_{lm} \max_{1\leq k \leq N}(\| v_k\|+\|v_i\|)~(\|
 v_j\|+\|v_i\|).
 \end{align*}
By Proposition \ref{prop 2.6}, the velocities have a uniform upper bound: $\|v_i(t)\|\leq \mathcal{V}(t)\leq \sqrt{N\E(0)}$, for any $i\in \{1,\ldots,N\}$. 
Thus, we obtain the desired result as follows:
\begin{align}
\label{eqn:rr31}
|I^{ij}_3|&\leq 4\max_{1\leq l,m \leq N } \psi_{lm}N\E(0).
 \end{align}

For the $|I_2^{ij}|$ case, we use the $|I_3^{ij}|$ result. Since the rotation operator $R_{x_i\shortrightarrow x_j}$ satisfies
\begin{align*}R_{x_i\shortrightarrow x_j}^{-1} = R_{x_i\shortrightarrow x_j}^{T} = R_{x_j\shortrightarrow x_i},\end{align*}
we reduce $I^{ij}_2$ to $I^{ji}_3$ by using the properties of the rotation operator as follows:
\begin{align*}
I^{ij}_2&=\Big\langle  R_{x_j\shortrightarrow x_i}\Big(\frac{dv_j}{dt}\Big),  R_{x_j\shortrightarrow x_i}(v_j)-v_i\Big\rangle = \Big\langle \frac{dv_j}{dt},  R_{x_i\shortrightarrow x_j}\Big( R_{x_j\shortrightarrow x_i}(v_j)-v_i\Big)\Big\rangle\\
&=
\Big\langle  \frac{dv_j}{dt},  v_j-R_{x_i\shortrightarrow x_j}(v_i)\Big\rangle = I^{ji}_3
\end{align*}
Therefore, the previous result \eqref{eqn:rr31} for  $I^{ij}_3$ implies that
\begin{align}
\label{eqn:rr41}
|I^{ij}_2|&\leq 4\max_{1\leq l,m \leq N } \psi_{lm}N\E(0).
 \end{align}

It now remains to obtain an estimate for $I^{ij}_1$:
\begin{align*}
I^{ij}_1=\Big\langle  \frac{dR(x_j,x_i)}{dt}\cdot v_j,  R_{x_j\shortrightarrow x_i}(v_j)-v_i\Big\rangle.
\end{align*}
We again take the absolute value of $I^{ij}_1$ and use the Cauchy inequality and triangle inequality to obtain
\begin{align*}
|I^{ij}_1|&\leq
\Big\| \frac{dR(x_j,x_i)}{dt}\cdot v_j\Big\|~\|R_{x_j\shortrightarrow x_i}(v_j)-v_i\| \leq \Big\| \frac{dR(x_j,x_i)}{dt}\cdot v_j\Big\|\big(\|R_{x_j\shortrightarrow x_i}(v_j)\|+\|v_i\|\big).
\end{align*}
From the properties of the rotation operator and the Cauchy inequality, it follows that
\begin{align*}\|R_{x_j\shortrightarrow x_i}(v_j)\|=\|v_j\| \hbox{ and } \Big\| \frac{dR(x_j,x_i)}{dt}\cdot v_j\Big\|\leq \Big\|\frac{dR(x_j,x_i)}{dt}\Big\|\|v_j\|.\end{align*}
Thus, it holds that
\begin{align*}
|I^{ij}_1|&\leq
 \Big\|\frac{dR(x_j,x_i)}{dt}\Big\|\|v_j\|(\|v_j\|+\|v_i\|) \leq 2 \Big\| \frac{dR(x_j,x_i)}{dt}\Big\|N\E(0).
\end{align*}
Here, we used  the uniform upper bound of velocities in Proposition \ref{prop 2.6}.
By Lemma \ref{lemma 4.3}, we have
\begin{align}
\label{eqn:rr51}
|I^{ij}_1|\leq C(N\E(0))^{3/2} +\frac{C}{\| x_i+x_j\|} (N\E(0))^{3/2},
\end{align}
where $C$ is a positive constant.

Finally, combining \eqref{eqn:rr31}, \eqref{eqn:rr41} and \eqref{eqn:rr51}, we obtain
\begin{align*}
\Big|\frac{d}{dt}\| R_{x_j\shortrightarrow x_i}(v_j)-v_i\|^2\Big|&\leq 2|I^{ij}_1|+2|I^{ij}_2|+2|I^{ij}_3|\\
&\leq
C\Big((N\E(0))^{3/2} +\frac{(N\E(0))^{3/2}}{\| x_i+x_j\|}
+\max_{1\leq i,k\leq N} \psi_{ik}N\E(0)
\Big).
\end{align*}
\end{proof}

We are ready to prove the velocity alignment result in the main theorem.

\begin{theorem}
\label{thm:43}
For $\psi$ satisfying \eqref{eqn:psi} and any given initial data $(x_i(0),v_i(0))_{i=1}^N $ satisfying the admissible condition in \eqref{eqn:adm}, the solution $(x_i,v_i)_{i=1}^N$ to \eqref{main} with $\sigma\geq0$  satisfies
\begin{align}
\label{eqn:43}
\lim_{t \to \infty} \|x_i(t)+x_j(t)\|\| R_{x_j(t)\shortrightarrow x_i(t)}(v_j(t))-v_i(t)\|=0.
\end{align}
\end{theorem}

\begin{proof}

We recall the identity in Proposition \ref{prop 2.6} for the case of $\sigma=0$.
\begin{align*}
\frac{d\E}{dt}=-
\sum_{i,j=1}^N\frac{\psi_{ij}(t)}{N^2}\| R_{x_j\shortrightarrow x_i}(v_j)-v_i\|^2.
\end{align*}
Taking the integral of the above with respect to time on $(0,t)$ yields
\begin{align}
\begin{aligned}\label{eq 5.2}
N^2\E(t)-N^2\E(0)
&=-\int_0^{t} \sum_{i,j=1}^N\psi_{ij}(\tau)\| R_{x_j(\tau)\shortrightarrow x_i(\tau)}(v_j(\tau))-v_i(\tau)\|^2d\tau.
\end{aligned}
\end{align}

For $f_{ij}(t) := \psi_{ij}(t)\| R_{x_j(t)\shortrightarrow x_i(t)}(v_j(t))-v_i(t)\|^2\geq 0$, since $f_{ij}(t)$ is nonnegative for any $i,j\in \{1,\ldots,N\}$ and $t>0$,
\eqref{eq 5.2} implies that
\begin{align}
\begin{aligned} \label{eqn:53}
\lim_{t\to \infty}\int_0^{t}f_{ij}(\tau)d\tau
\leq \lim_{t\to \infty}\int_0^{t}\sum_{i=1}^N \sum_{j=1}^N f_{ij}(\tau)d\tau = \lim_{t\to \infty}\Big(N^2\E(0)-N^2\E(t)\Big)
\leq N^2\E(0).
\end{aligned}
\end{align}
 Recall that if $x_i(t)+x_j(t)= 0$, we have $f_{ij}(t)=0.$

We assume that  $x_i(t)+x_j(t)\ne 0$ and  take the absolute value of the derivative of $f_{ij}(t)$ with respect to $t$ to obtain
\begin{align*}
\Big|\frac{df_{ij}(t)}{dt}\Big|&= \bigg| \frac{\langle v_i(t)-v_j(t),x_i(t)-x_j(t)\rangle}{\|x_i(t)-x_j(t)\|}\psi'(\|x_i(t)-x_j(t)\|)~ \| R_{x_j(t)\shortrightarrow x_i(t)}(v_j(t))-v_i(t)\|^2\\
&\quad+ \psi_{ij}\frac{d}{dt}\| R_{x_j(t)\shortrightarrow x_i(t)}(v_j(t))-v_i(t)\|^2\bigg|\\
&\leq
\Big|\frac{\langle v_i(t)-v_j(t),x_i(t)-x_j(t)\rangle}{\|x_i(t)-x_j(t)\|}\psi'(\|x_i(t)-x_j(t)\|) ~ \| R_{x_j(t)\shortrightarrow x_i(t)}(v_j(t))-v_i(t)\|^2\Big|\\
&\quad+  \Big|\psi_{ij}\frac{d}{dt}\| R_{x_j(t)\shortrightarrow x_i(t)}(v_j(t))-v_i(t)\|^2\Big| =: K^{ij}_1+K^{ij}_2.
\end{align*}

We estimate $K^{ij}_1$ as follows:
\begin{align*}
K_1^{ij}
&\leq
\Big|\frac{\langle v_i(t)-v_j(t),x_i(t)-x_j(t)\rangle}{\|x_i(t)-x_j(t)\|}\Big| ~|\psi'(\|x_i(t)-x_j(t)\|)|~
\|R_{x_j(t)\shortrightarrow x_i(t)}(v_j(t))-v_i(t)\|^2.
\end{align*}
From Proposition~\ref{prop 2.6}, we have
\begin{align*}
\Big|\frac{\langle v_i(t)-v_j(t),x_i(t)-x_j(t)\rangle}{\|x_i(t)-x_j(t)\|}\Big|
&\leq \| v_i(t)-v_j(t)\|
\leq  2\mathcal{V}(t) \leq 2\sqrt{N\E(0)}.
\end{align*}
Similarly, Proposition~\ref{prop 2.6} and Lemma~\ref{lemma 2.4} yield that
\begin{align*}
\| R_{x_j(t)\shortrightarrow x_i(t)}(v_j(t))-v_i(t)\|^2&\leq
(\| R_{x_j(t)\shortrightarrow x_i(t)}(v_j(t))\|+\|v_i(t)\|)^2 = (\| v_j(t)\|+\|v_i(t)\|)^2 \leq 4N\E(0).
\end{align*}
From the assumption on $\psi$, we conclude that
\begin{align}
\label{eqn:61}
K^{ij}_1&\leq 8C(N\E(0))^{3/2}.
\end{align}

Next, we estimate $K^{ij}_2$. 
Lemma~\ref{lem:psi} implies that
\begin{align*}
K^{ij}_2&= \Big|\psi(\|x_i(\tau)-x_j(\tau)\|)\frac{d}{d\tau}\| R_{x_j(\tau)\shortrightarrow x_i(\tau)}(v_j(\tau))-v_i(\tau)\|^2\Big|,\\
&\leq C~\|x_i(\tau)+x_j(\tau)\|^2\Big|\frac{d}{d\tau}\| R_{x_j(\tau)\shortrightarrow x_i(\tau)}(v_j(\tau))-v_i(\tau)\|^2\Big|.
\end{align*}
On the other hand, Lemma \ref{lemma 3.4} and $\|x_i(\tau)+x_j(\tau)\| \leq 2$ yield
\begin{align*}
&\|x_i(\tau)+x_j(\tau)\|   ~\Big|\frac{d}{dt}\| R_{x_j\shortrightarrow x_i}(v_j)-v_i\|^2\Big|
\\&\leq
C\|x_i(\tau)+x_j(\tau)\|\Big((N\E(0))^{3/2} +\frac{ (N\E(0))^{3/2}}{\| x_i(\tau)+x_j(\tau)\|}
+\max_{1\leq i,k\leq N} \psi_{ik}N\E(0)
\Big)\\&\leq
 C_1
.
\end{align*}
Thus, it follows from the above estimates that
\begin{align}
\label{eqn:62}
K^{ij}_2&\leq 2 C C_1.
\end{align}
From \eqref{eqn:61} and \eqref{eqn:62}, for any $t>0$, it holds that
\begin{align}
\label{eqn:54}
f(t)=0\quad  \hbox{ or } \quad \Big|  \frac{df_{ij}(t)}{dt}\Big|\leq C
\end{align}
where $C$ is a positive constant.

From \eqref{eqn:53} and \eqref{eqn:54}, we can apply Lemma \ref{lemma 3.3} to $f_{ij}(t)$ to obtain that
\begin{align*}\lim_{t\to \infty }\psi(\|x_i(t)-x_j(t)\|)\| R_{x_j(t)\shortrightarrow x_i(t)}(v_j(t))-v_i(t)\|^2=0.\end{align*}
From Lemma~\ref{lem:psi}, we conclude \eqref{eqn:43}.
\end{proof}


Next, we focus on the flocking theorem. As a direct consequence of Proposition~\ref{prop 2.6}, we have the following inequalities.

\begin{lemma}
\label{lem:bddo}
Let $(x_i,v_i)_{i=1}^N$ be the solution to \eqref{main} and $\sigma >0$. For $t>0$ and $1 \leq i,j \leq N$, it holds that
\begin{align}
\label{eqn:bddo}
\|v_i(t) \|^2 \leq  N\E(0)  \quad\hbox{ and } \quad \|x_i(t) - x_j(t)\|^2 \leq \frac{2N^2 \E(0)}{\sigma}.
\end{align}
\end{lemma}

\begin{proposition}
\label{prop:lip}
Let $(x_i,v_i)_{i=1}^N$ be the solution to \eqref{main}. If
\[2\sigma > N^2\E(0),\]
 then $\displaystyle\frac{d}{dt}v_i$ and $\displaystyle \frac{d}{dt}R_{x_j\shortrightarrow x_i}(v_j)$ are  bounded.
\end{proposition}

\begin{proof}

Since $x_i \in \mathbb{S}^2$, by Lemma~\ref{lem:bddo},
$\displaystyle -\frac{\|v_i\|^2}{\|x_i\|^2}x_i$  is  bounded.
The orthogonality of the rotation operator in Lemma~\ref{lemma 2.3} and Lemma~\ref{lem:bddo} imply that
\begin{align*}
\| R_{x_j\shortrightarrow x_i}(v_j) - v_i \| \leq   \| v_j \| +  \|v_i \| \leq 2\sqrt{N\E(0)}.
\end{align*}
As $\psi_{ij}$ is bounded, we conclude that the second term of $\dot{v}_i$ in \eqref{main} is bounded. The boundedness of the last term follows from $x_i \in \mathbb{S}^2$.

Next, we prove that $\displaystyle \frac{d}{dt}R_{x_j\shortrightarrow x_i}(v_j)$ is bounded. By the chain rule, it holds that
\begin{align}
\label{eqn:lip21}
\frac{d}{dt}(R_{x_j\shortrightarrow x_i}(v_j)) = R_{x_j\shortrightarrow x_i}\frac{d}{dt}v_j + v_j\frac{d}{dt}R_{x_j\shortrightarrow x_i}.
\end{align}
Similar to $v_i$, the first term in \eqref{eqn:lip21} is uniformly bounded. From the direct computation of the second term in \eqref{eqn:lip21}, it follows that
\begin{align}\begin{aligned}
\label{eqn:lip22}
\frac{d}{dt}(R_{x_j\shortrightarrow x_i}) &=
\frac{d}{dt} \left\{\langle x_{j},x_{i}\rangle  I - x_{j} x_{i}^T + x_{i} x_{j}^T \right\} + \frac{1}{1 +  \langle x_{j},x_{i}\rangle } \frac{d}{dt}  \left\{ (x_{j} \times x_{i})(x_{j} \times x_{i})^T \right\} \\
&\quad + \frac{d}{dt} \left( \frac{1}{1 +   \langle x_{j},x_{i}\rangle } \right) (x_{j} \times x_{i})(x_{j} \times x_{i})^T.
\end{aligned}
\end{align}
The uniform boundedness of $(x_i,v_i)_{i=1}^N$ shows that of the first term in the right hand side of \eqref{eqn:lip22}.

Moreover, by Lemma~\ref{lem:bddo}, we have
\begin{align*}
1+ \langle x_{i},x_{j}\rangle  = \frac{1}{2}\|x_i + x_j\|^2 > 2 - \frac{N\E(0)}{\sigma} > 0.
\end{align*}
Combining this with uniform boundedness of $(x_i,v_i)_{i=1}^N$, we conclude that the second term and the third term in \eqref{eqn:lip22} are uniformly bounded. Therefore,  $\displaystyle \frac{d}{dt}R_{x_j\shortrightarrow x_i}(v_j)$ is bounded.
\end{proof}

\begin{theorem}\label{thm 4.8}Let $(x_i,v_i)_{i=1}^N$ be the solution to \eqref{main} and $\psi$ satisfy \eqref{eqn:psi}. If $2\sigma > N^2\E(0)$, then \eqref{main} has time-asymptotic flocking on a unit sphere.
\end{theorem}

\begin{proof}
From the assumption on $\psi$ and Lemma \ref{lem:bddo},
\[\sup_{t\geq 0}\max_{1\leq i,j\leq N }\|x_i(t) - x_j(t)\|<2,\]
and
\[\frac{\psi_{ij}}{N}>\psi(\sqrt{2N^2 \E(0)/\sigma})=:C_{\psi},\]
for any $i,j\in\{1,\ldots,N\}$ and $t\geq 0$. By Proposition~\ref{prop 2.6}, it holds that
\begin{align*}
C_{\psi} \int_0^t   \| R_{x_j(s) \rightarrow x_i(s)}(v_j(s))-v_i(s)\|^2 ds \leq \E(0) \quad \hbox{ for all } t \in [0, \infty).
\end{align*}
Therefore, $ \sum_{i,j=1}^{N} \| R_{x_j(t) \rightarrow x_i(t)}(v_j(t))-v_i(t)\|^2$ is integrable in $[0, \infty)$.

As $\displaystyle\frac{d}{dt}v_i$ and $\displaystyle\frac{d}{dt}R_{x_j\shortrightarrow x_i}(v_j)$ are  bounded in $[0, \infty)$ for any $i,j\in \{1,\ldots,N\}$ from Lemma~\ref{lem:bddo} and Proposition~\ref{prop:lip}, $ \displaystyle \sum_{i,j=1}^{N} \frac{d}{dt}\| R_{x_j(t) \rightarrow x_i(t)}(v_j(t))-v_i(t)\|^2$ is bounded. From Lemma~\ref{lemma 3.3}, we conclude that
\begin{align*}
\sum_{i,j=1}^{N} \| R_{x_j(t) \rightarrow x_i(t)}(v_j(t))-v_i(t)\|^2\rightarrow 0 \quad \hbox{ as }\quad  t \rightarrow \infty.
\end{align*}
\end{proof}

\section{Conclusion}
\setcounter{equation}{0}
 In this paper, we consider a C-S type flocking model on the unit sphere $\bbs^2$. To derive the flocking model on the unit sphere, we introduced the rotation operator $R$. The rotation operator  has the modulus conservation property, but singularity occurs at antipodal points. Therefore, in order to cancel such singularity, we assumed that $\phi$ vanished at antipodal points. Moreover, we introduce a centripetal force term to obtain  conservation of the modulus, $\|x_i\|\equiv 1$. For the velocity alignment, we employed energy dissipation and a Barbalat's lemma type argument. To define flocking on a spherical surface, in addition to velocity alignment, we consider antipodal points avoidance that the positions of two agents are not located at antipodal points, which corresponds to the boundedness of the relative position of the flocking definition in the flat space. However, it is difficult to control the relative position by the geometric properties of the spherical surface. Therefore, in addition to relative velocity, a bonding force term is added. Using the bonding force term and initial conditions, we controlled the relative positions of the particles.

  \appendix

\section{Proof of Lemma ~\ref{lemma 3.3}}\label{ap:0}
In this section, we present the proof of Lemma~\ref{lemma 3.3}.
Suppose that \eqref{eqn:2lem33} does not hold. Then
\begin{align*}\limsup_{\tau\to \infty} f(\tau) > 0,\end{align*}
and there is a sequence of real numbers $\{b_n\}$ such that $b_n \to \infty$ as $n\to \infty$ and
\begin{align*}\lim_{n\to \infty} f(b_n) = d\end{align*}
for some $d>0$.
We can choose $b_n$ on the support $\{t\in [a,\infty):f(t)>0\}$ of $f$ and 
there is $N\in \bbn$ such that for $n>N$,
\begin{align*}f(b_n)>\frac{d}{2}.\end{align*}
From \eqref{eqn:11lem33}, it holds that 
\begin{align}\label{eq 2.8}
0<-C|\tau-b_n|+f(b_n) <f(\tau) \mbox{ for all } n>N  \hbox{ and } \tau \in \left(b_n - \frac{d}{2C},~ b_n + \frac{d}{2C} \right).
\end{align}
By replacing $\{b_n\}$ with its subsequence if necessary,
we assume that
\begin{align}
\label{eqn:lem3313}
b_{n+1}-b_n>\frac{d}{C}.
\end{align}

We define a sequence $\{a_n\}$ such that
\begin{align*}a_n:=\int_{b_n-\frac{d}{2C}}^{b_{n+1}-\frac{d}{2C}} f(\tau)d\tau.\end{align*}
Since $f$ is nonnegative and $\{b_n\}$ satisfies \eqref{eqn:lem3313}, we have
\begin{align*}a_n=\int_{b_n-\frac{d}{2C}}^{b_{n+1}-\frac{d}{2C}} f(\tau)d\tau>\int_{b_n-\frac{d}{2C}}^{b_n+\frac{d}{2C}}f(\tau)d\tau,\end{align*}
and  we obtain the lower bound of $a_n$ by using \eqref{eq 2.8}.
\begin{align*}a_n>\int_{b_n-\frac{d}{2C}}^{b_n+\frac{d}{2C}}\Big(-C|\tau-b_n|+f(b_n)\Big)d\tau=\frac{4 f(b_n)d-d^2}{4C}>\frac{d^2}{4C}.\end{align*}
By  definition of $\{a_n\}$, the following holds.
\begin{align*}\lim_{b\to \infty}\int_a^b f(\tau)d\tau=\sum a_n> \lim_{n\to \infty} \frac{d^2}{4C}n=\infty,\end{align*}
which contradicts \eqref{eqn:1lem33} and thus we conclude \eqref{eqn:2lem33}.

\section{Proof of Lemma~\ref{lemma 4.3}}
\label{ap:1}

In this section, we prove Lemma~\ref{lemma 4.3}. First, we show the following elementary estimates.

\begin{lemma}\label{aux lemma for dR}
For any $(x_i,v_i)_{i=1}^N$ satisfying $\dot{x}_i=v_i$ for all $i \in \{1,2, \cdots, N\}$, we assume that
\begin{align*}\|x_i\|=1,\quad \langle v_i,x_i\rangle=0,\quad  \|v_i\|\leq \mathcal{V}_{max} \hbox{ and } x_i + x_j \ne 0 \hbox{ for any } i,j \in \{1,2, \cdots, N\}.\end{align*}
Then, the following estimate holds:
\begin{align}
\label{eqn:dr10}
\bigg\|\frac{d}{dt}\bigg[\frac{1}{1+\langle x_i,x_j\rangle}(x_i\times x_j)(x_i\times x_j)^T\bigg]\bigg\|\leq \frac{C}{\| x_i+x_j\|} \mathcal{V}_{max},
\end{align}
for a positive constant $C$.
\end{lemma}
\begin{proof}
We take the time-derivative of the given term. Elementary differentiation shows that
\begin{align*}
&\frac{d}{dt}\bigg(\frac{1}{1+\langle x_i,x_j\rangle}(x_i\times x_j)(x_i\times x_j)^T\bigg)\\
&=\frac{\langle v_i,x_j\rangle+\langle x_i,v_j\rangle}{(1+\langle x_i,x_j\rangle)^2}(x_i\times x_j)(x_i\times x_j)^T
+\frac{1}{1+\langle x_i,x_j\rangle}(v_i\times x_j)(x_i\times x_j)^T
\\
&\qquad+\frac{1}{1+\langle x_i,x_j\rangle}(x_i\times v_j)(x_i\times x_j)^T
+\frac{1}{1+\langle x_i,x_j\rangle}(x_i\times x_j)^T(v_i\times x_j)^T
\\
&\qquad+\frac{1}{1+\langle x_i,x_j\rangle}(x_i\times x_j)^T(x_i\times v_j)^T
\\&=:J_1+J_2+J_3+J_4+J_5.\end{align*}

We claim that
\begin{align}
\label{eqn:dr11}
|J_1|&\leq \frac{C}{\|x_i+x_j\|} \mathcal{V}_{max}
\end{align}
for some constant $C>0$. Note that for any vector $x\in \bbr^3$, $x\times x=0$. Thus the following holds:
\begin{align}\label{eq 3.1}
x\times y  =   (x+y)\times y,~\mbox{for}~ x,y\in\bbr^3.
\end{align}
Moreover, for each $i\in \{1,\ldots,N\}$, the velocity and position are orthogonal, i.e.,
\begin{align}\label{eq 3.2}
\langle v_i,x_i\rangle=0, ~ \mbox{for all $i\in \{1,\ldots,N\}$}.
\end{align}
\eqref{eq 3.1} and \eqref{eq 3.2} yield
\begin{align*}
J_1&=
\frac{\langle v_i,x_j\rangle+\langle x_i,v_j\rangle}{(1+\langle x_i,x_j\rangle)^2}(x_i\times x_j)(x_i\times x_j)^T\\
&=
\frac{\langle v_i,x_j\rangle+\langle x_i,v_j\rangle}{(1+\langle x_i,x_j\rangle)^2}\big((x_i+x_j) \times x_j\big)\big((x_i+x_j)\times x_j\big)^T
\\
&=
\frac{\langle v_i,x_j+x_i\rangle+\langle x_i+x_j,v_j\rangle}{(1+\langle x_i,x_j\rangle)^2}\big((x_i+x_j) \times x_j\big)\big((x_i+x_j)\times x_j\big)^T.\end{align*}
The Cauchy inequality and triangle inequality show that
\begin{align*}
\|J_1\|&\leq&
\frac{C}{(1+\langle x_i,x_j\rangle)^2}\Big(\|v_i\|~\|x_i+x_j\|+\|x_i+x_i\|~\|v_j\|\Big)\|x_i+x_j\|~\|x_j\|~\|x_i+x_j\|~\|x_j\|.
\end{align*}
Here, $C$ is a positive constant.  By  assumption,  $x_i$ and $x_j$ are unit vectors  and $v_i$, $v_j$ are bounded such that
\begin{align*}\|v_i\|, \|v_j\|\leq \mathcal{V}_{max}.\end{align*}
 This implies the following estimate for $J_1$:
\begin{align*}
\|J_1\|&\leq
\frac{C}{(1+\langle x_i,x_j\rangle)^2}\Big(\|v_i\|~\|x_i+x_j\|+\|x_i+x_j\|~\|v_j\|\Big)\|x_i+x_j\|~\|x_j\|~\|x_i+x_j\|~\|x_j\|\\
&\leq
\frac{C}{(1+\langle x_i,x_j\rangle)^2}\big(\|v_i\|+\|v_j\|\big)\|x_i+x_j\|^3\\
&\leq
\frac{C}{(1+\langle x_i,x_j\rangle)^2} \mathcal{V}_{max}\|x_i+x_j\|^3.
\end{align*}
Since we have  $\displaystyle 1+\langle x_i,x_j\rangle=\frac{\|x_i+x_j\|^2}{2}$, we obtain the estimate for $J_1$ in \eqref{eqn:dr11}.

Next, we show an estimate for $J_2$ as
\begin{align}
\label{eqn:dr12}
\|J_2\|\leq \frac{C}{\| x_i+x_j\|} \mathcal{V}_{max}.
\end{align}
Here, we take the matrix norm of $J_2$.
By \eqref{eq 3.1},
\begin{align*}
J_2=\frac{1}{1+\langle x_i,x_j\rangle}(v_i\times x_j)(x_i\times x_j)^T=\frac{1}{1+\langle x_i,x_j\rangle}(v_i\times x_j)\big((x_i+x_j)\times x_j\big)^T.
\end{align*}
Taking the matrix norm on the above equation, we obtain
\begin{align*}
\|J_2\|=\Big\|\frac{1}{1+\langle x_i,x_j\rangle}(v_i\times x_j)\big((x_i+x_j)\times x_j\big)^T\Big\|\leq
\frac{C}{1+\langle x_i,x_j\rangle} \|v_i\|~\|x_j\|~\|x_i+x_j\|~\|x_j\|.
\end{align*}
Then, \eqref{eqn:dr12} follows from boundedness of $x_i$, $x_j$, $v_i$ and $v_j$, and the simple fact that
 $$\displaystyle 1+\langle x_i,x_j\rangle=\frac{\|x_i+x_j\|^2}{2}.$$

From the parallel argument, the following estimates hold for $\|J_{3}\|,~\|J_{4}\|$ and $\|J_{5}\|$:
\begin{align}
\label{eqn:dr13}
\|J_{3}\|,~\|J_{4}\|,~\|J_{5}\|\leq \frac{C}{\| x_i+x_j\|} \mathcal{V}_{max}.
\end{align}
We omit the proof of the above estimates. 
By adding \eqref{eqn:dr11}, \eqref{eqn:dr12} and \eqref{eqn:dr13}, we conclude \eqref{eqn:dr10}.

\end{proof}

The result in Lemma \ref{aux lemma for dR} is optimal as in the following example.

\begin{example}
Consider
\begin{align*}
x_1(t)&=\left(t^2 \sin \frac{1}{\sqrt{t}},~-t^2 \cos \frac{1}{\sqrt{t}},~-\sqrt{-\left(t^2 \sin \frac{1}{\sqrt{t}}\right)^2-\left(-t^2 \cos \frac{1}{\sqrt{t}}\right)^2+1}\right),\\
x_2(t)&=(0,0,1).
\end{align*}
Then
\begin{align*}
\frac{1}{1+\langle x_i,x_j\rangle}(x_i\times x_j)(x_i\times x_j)^T=\left(
\begin{matrix}
 \left(\sqrt{1-t^4}+1\right) \cos ^2\frac{1}{\sqrt{t}} & \frac{t^4 \sin \frac{2}{\sqrt{t}}}{2-2 \sqrt{1-t^4}} & 0 \\
 \frac{t^4 \sin \frac{2}{\sqrt{t}}}{2-2 \sqrt{1-t^4}} & \left(\sqrt{1-t^4}+1\right) \sin^2\frac{1}{\sqrt{t}} & 0 \\
 0 & 0 & 0
\end{matrix}
\right)
\end{align*}
and
\begin{align*}
&\hspace{-2em}\frac{d}{dt}\bigg[\frac{1}{1+\langle x_i,x_j\rangle}(x_i\times x_j)(x_i\times x_j)^T\bigg]\\&=
\left(\begin{matrix}
 \frac{\left(\sqrt{1-t^4}+1\right) \sin \frac{2}{\sqrt{t}}}{2 t^{3/2}}-\frac{2 t^3 \cos^2\frac{1}{\sqrt{t}}}{\sqrt{1-t^4}} & m(t)   & 0 \\
  m(t)& -\frac{2 t^3\sin^2\frac{1}{\sqrt{t}} }{\sqrt{1-t^4}}-\frac{\left(\sqrt{1-t^4}+1\right) \sin \frac{2}{\sqrt{t}}}{2 t^{3/2}} & 0 \\
 0 & 0 & 0
\end{matrix}\right),\end{align*}
where
\begin{align*}m(t)=\frac{2 t^3 \left(t^4+2 \sqrt{1-t^4}-2\right) \sin \frac{2}{\sqrt{t}}-t^{5/2} \left(t^4+\sqrt{1-t^4}-1\right) \cos \frac{2}{\sqrt{t}}}{2 \sqrt{1-t^4} \left(\sqrt{1-t^4}-1\right)^2}.\end{align*}
Note that $\|x_1(t)+x_2(t)\|=\sqrt{2-2 \sqrt{1-t^4}}=t^2+O(t^{4})$, \begin{align*}\|v_1(t)\|=\frac{1}{2} \sqrt{\frac{t \left(t^4-16 t-1\right)}{t^4-1}}=\frac{\sqrt{t}}{2}+O(t^{\frac{3}{2}}),\end{align*}
and
\begin{align*}\bigg\|\frac{d}{dt}\bigg[\frac{1}{1+\langle x_i,x_j\rangle}(x_i\times x_j)(x_i\times x_j)^T\bigg]\bigg\|=\Big(\frac{1}{t^{\frac{3}{2}}}+O(t^{\frac{3}{2}})\Big)\sin\frac{2}{\sqrt{t}}+\Big(\frac{1}{t^{\frac{3}{2}}}+O(t^{\frac{3}{2}})\Big)\cos\frac{2}{\sqrt{t}}.\end{align*}
Thus, we conclude that the result in Lemma \ref{aux lemma for dR} is optimal:
\begin{align*}
\bigg\|\frac{d}{dt}\bigg[\frac{1}{1+\langle x_i,x_j\rangle}(x_i\times x_j)(x_i\times x_j)^T\bigg]\bigg\|\sim \frac{\mathcal{V}_{max}}{\| x_i+x_j\|}.
\end{align*}
\end{example}

\medskip

\begin{proof}[{\bf Proof of Lemma~\ref{lemma 4.3}}]
From Definition \ref{mat rot}, the rotation matrix is given by
\begin{align*}
R(x_j,x_i)
&=\langle x_i,x_j\rangle I+x_i x_j^T-x_j x_i^T +\frac{1}{1+\langle x_i,x_j\rangle}(x_i\times x_j)(x_i\times x_j)^T.
\end{align*}
If we take the time-derivative of the above, then
\begin{align*}
\frac{dR(x_j,x_i)}{dt}
&=\langle v_i,x_j\rangle I+\langle x_i,v_j\rangle I
+v_i x_j^T +x_i  v_j^T\\
&-v_j x_i^T -x_j  v_i^T
+\frac{d}{dt}\bigg[\frac{1}{1+\langle x_i,x_j\rangle}(x_i\times x_j)(x_i\times x_j)^T\bigg].
\end{align*}

Taking the matrix norm $\|\cdot\|=\|\cdot\|_{2}$ leads to
\begin{align*}
\Big\|\frac{dR(x_j,x_i)}{dt}\Big\|
&=\Big\|\langle v_i,x_j\rangle I+\langle x_i,v_j\rangle I
+v_i x_j^T +x_i  v_j^T\\
&\quad-v_j x_i^T -x_j  v_i^T
+\frac{d}{dt}\bigg[\frac{1}{1+\langle x_i,x_j\rangle}(x_i\times x_j)(x_i\times x_j)^T\bigg]\Big\|
\\
&\leq
 \|\langle v_i,x_j\rangle I\|+\|\langle x_i,v_j\rangle I\|
+
\|v_i x_j^T \|
+
\|x_i  v_j^T\|
\\
&\quad
+\|v_j x_i^T \|
 +
 \|x_j  v_i^T\|
 +
 \bigg\|\frac{d}{dt}\bigg[\frac{1}{1+\langle x_i,x_j\rangle}(x_i\times x_j)(x_i\times x_j)^T\bigg]\bigg\|
.
\end{align*}
The six terms except for the last term on the right-hand side of the above inequality is bounded by $\mathcal{V}_{max}$:
\begin{align}\label{eq 3.3}
 \|\langle v_i,x_j\rangle I\|,~\|\langle x_i,v_j\rangle I\|,~\|v_i x_j^T \|,~\|x_i  v_j^T\|,~\|v_j x_i^T \|,~ \|x_j  v_i^T\| \leq C \mathcal{V}_{max},
\end{align}
where $C$ is a positive constant.
From  Lemma \ref{aux lemma for dR}, it follows that
\begin{align}\label{eq 3.4}
\bigg\|\frac{d}{dt}\bigg[\frac{1}{1+\langle x_i,x_j\rangle}(x_i\times x_j)(x_i\times x_j)^T\bigg]\bigg\|\leq \frac{C}{\| x_i+x_j\|} \mathcal{V}_{max}.
\end{align}
\eqref{eq 3.3} and \eqref{eq 3.4} imply the desired result.
\end{proof}~
\newline

\section{Equivalence of the flocking definitions}

\begin{lemma}
\label{lem:equ}Let
\begin{align*} f_k(t):= \|x_i(t) + x_j(t) \| \|  R_{x_j(t)\shortrightarrow x_i(t)}(v_j(t)) - v_i(t) \|^k.\end{align*}
Suppose that $v_i$ and $v_j$ are uniformly bounded in time and $x_i, x_j \in \mathbb{S}^2$. Then, the following conditions
\begin{align*}
\lim_{t \rightarrow \infty} f_k(t) = 0
\end{align*}
are equivalent for any $k>0$.
\end{lemma}

\begin{proof}
Choose $k > m > 0$. As
\begin{align*}
f_k = \|  R_{x_j\shortrightarrow x_i}(v_j) - v_i \|^{k-m} f_m \hbox{ and } f_m^{\frac{1}{m}} = \|x_i + x_j \|^{\frac{1}{m} - \frac{1}{k}} f_k^{\frac{1}{k}}
\end{align*}
hold, it is enough to check that $\|x_i + x_j \|$ and $\|  R_{x_j\shortrightarrow x_i}(v_j) - v_i \|$ are uniformly bounded in time. As $x_i$ and $x_j$ are in $\mathbb{S}^2$, we have $\|x_i + x_j \| \leq 2$. On the other hand, from Lemma~\ref{lemma 2.4},
\begin{align*}
\|  R_{x_j\shortrightarrow x_i}(v_j) - v_i \| \leq \|  R_{x_j\shortrightarrow x_i}(v_j) \| + \|v_i \| = \|v_j\| + \|v_i\|.
\end{align*}
As $v_i$ and $v_j$ are uniformly bounded in time, we conclude.
\end{proof}

\section*{Acknowledgments}
S.-H. Choi is partially supported by NRF of Korea (no. 2017R1E1A1A03070692) and Korea Electric Power Corporation(Grant number: R18XA02).

\bibliographystyle{amsplain}

\end{document}